\documentclass[journal,twoside,web]{ieeecolor}
\definecolor{darkblue}{rgb}{0.0,0.4,1.0}
\usepackage[backend=bibtex]{biblatex}
\usepackage{generic}
\usepackage{amsmath,amssymb,amsfonts}
\usepackage{graphicx}
\DeclareGraphicsExtensions{.eps,.pdf,.png,.jpg}
\usepackage{algorithm}
\usepackage{hyperref}
\hypersetup{hidelinks=true}
\usepackage{textcomp}
\def\BibTeX{{\rm B\kern-.05em{\sc i\kern-.025em b}\kern-.08em
    T\kern-.1667em\lower.7ex\hbox{E}\kern-.125emX}}
\usepackage{amsmath} 
\usepackage{amssymb}
\usepackage{graphicx}
\usepackage{graphics}
\usepackage{cleveref}
\usepackage{algpseudocode}
\usepackage{subcaption}

\newcommand{\ssqcqp}[1]{SS-QCQP(#1)}

\newtheorem{assumption}{Assumption}
\newtheorem{theorem}{Theorem}
\newtheorem{lemma}{Lemma}

\newtheorem{definition}{Definition}
\newtheorem{remark}{Remark}

\newcommand{\revise}[1]{#1}

\begin{document}
\title{Anytime-Feasible First-Order Optimization via Safe Sequential QCQP}
%

\author{Jiarui Wang and Mahyar Fazlyab 
\thanks{The authors are with the Department of Electrical and Computer Engineering, Johns Hopkins University, Baltimore, MD 21218 USA (e-mail: \{jwang486, mahyarfazlyab\}@jhu.edu).}}

\maketitle


\begin{abstract}
This paper presents the Safe Sequential Quadratically Constrained Quadratic Programming (SS-QCQP) algorithm, a first-order method for smooth inequality-constrained nonconvex optimization that guarantees feasibility at every iteration.  
The method is derived from a continuous-time dynamical system whose vector field is obtained by solving a convex QCQP that enforces monotonic descent of the objective and forward invariance of the feasible set.  
\revise{The resulting continuous-time dynamics achieve an $O(1/t)$ ergodic convergence rate for the stationarity measure under standard constraint qualification conditions.}
We then propose a safeguarded Euler discretization with adaptive step-size selection that preserves this convergence rate while maintaining both descent and feasibility in discrete time.  
To enhance scalability, we develop an active-set variant (SS-QCQP-AS) that selectively enforces constraints near the boundary, substantially reducing computational cost without compromising theoretical guarantees.  
Numerical experiments on a multi-agent nonlinear optimal control problem demonstrate that SS-QCQP and SS-QCQP-AS maintain feasibility, exhibit the predicted convergence behavior, and deliver solution quality comparable to second-order solvers such as SQP and IPOPT.  
\end{abstract}


\section{INTRODUCTION}
Constrained optimization forms the foundation of numerous engineering disciplines, including optimal power flow~\cite{erseghe2014distributed}, optimal control~\cite{bagherzadeh2021guaranteed}, and robotics~\cite{wensing2023optimization}. In many practical settings, particularly when algorithms are deployed on high-frequency physical systems such as quadrotors, optimization must be executed in real time~\cite{nguyen2024tinympc}\cite{bishop2024relu}. This necessitates producing feasible solutions prior to full convergence while operating under strict computational limitations. Meeting these requirements poses fundamental challenges for the design of optimization methods that are simultaneously efficient, reliable, and capable of enforcing constraints at all intermediate iterates.

A key requirement in this setting is \emph{anytime feasibility}. This property guarantees that the algorithm can be safely interrupted at any time while still returning a feasible, though possibly suboptimal, solution. Such a property is essential in real-world deployments, where constraints frequently encode safety conditions; violating them may induce instability, equipment damage, or safety hazards.

Another desirable property is \emph{monotonic descent}, meaning that the objective value decreases at every iteration. Since premature termination is common in real-time settings, monotonic descent ensures consistent improvement over time and provides a quantifiable performance guarantee for any intermediate solution.

A third desirable property is \emph{scalability}. Modern machine learning, control, and robotics applications increasingly involve high-dimensional decision variables and large numbers of constraints, for example, model predictive control with long horizons, multi-agent coordination tasks, optimization problems arising in real-time perception and planning,  or data-driven verification~\cite{dietrich2024nonconvex}.
In such settings, each iteration must be computationally inexpensive, ideally relying only on first-order information, so that the algorithm can meet strict real-time deadlines and be deployed on resource-limited hardware.

Despite the importance of these properties, most classical algorithms do not satisfy them simultaneously. 
\revise{Although some primal interior-point methods can maintain feasibility when
initialized from a strictly feasible point, widely used primal-dual interior-
point solvers such as IPOPT may allow infeasible intermediate iterates and enforce feasibility only at convergence~\cite{nemirovski2008interior,nocedal2009adaptive,wachter2002interior}, which is incompatible with safety-critical applications.
}
Sequential Quadratic Programming (SQP) methods~\cite{panier1987superlinearly,schittkowski1983convergence,panier1993combining} may violate constraints even when initialized at a feasible point due to the mismatch between linearized and nonlinear constraints. Penalty and augmented Lagrangian methods deliberately relax constraints and therefore cannot guarantee feasibility during intermediate iterations. Moreover, many of these classical methods require second-order information or the solution of large linear systems at each iteration, which limits scalability and hinders deployment in real-time environments. These limitations highlight the need for first-order optimization algorithms that simultaneously ensure anytime feasibility, monotonic descent, and scalability.

\subsubsection*{Our Contribution}
To address the limitations of existing constrained optimization methods, this paper proposes a new first-order framework for solving general inequality-constrained nonconvex problems. Our main contributions are as follows: (1) Adopting a control-theoretic perspective, we design a continuous-time dynamical system whose vector field is obtained by solving a convex quadratically constrained quadratic program (QCQP) that uses only first-order information. The dynamics guarantee forward invariance of the feasible set, monotonic descent, and global convergence to Karush–Kuhn–Tucker (KKT) points. \revise{We further show that the continuous-time system achieves an $O(1/t)$ ergodic convergence rate in
terms of a first-order stationary measure.} (2) We discretize the continuous-time dynamics using a safeguarded Armijo-type line search. The discrete-time method preserves anytime feasibility and monotonic decrease of the objective and matches the continuous-time $O(1/t)$ rate.
(3) To enhance scalability in problems with many constraints, we introduce an active-set variant that enforces only a subset of nearly active constraints at each iteration, substantially reducing per-iteration computational cost.
(4) We establish convergence guarantees for both the full and active-set variants, showing that each achieves an $O(1/k)$ ergodic convergence rate in discrete time under standard Mangasarian-Fromovitz constraint qualification (MFCQ) and Lipschitz assumptions, matching the rate of classical gradient descent for unconstrained problems.
(5) We provide detailed implementation guidelines, including conic reformulations, solver considerations, and adaptive hyperparameter selection, to ensure strong empirical performance in practical deployments.

\subsection{Related Work}
\subsubsection{SQP Methods}
A widely used approach for solving nonconvex constrained optimization problems is to approximate them by a sequence of convex subproblems. The most prominent example is Sequential Quadratic Programming (SQP), which computes search directions by solving a convex quadratic program (QP) obtained from a second-order expansion of the Lagrangian and a first-order linearization of the constraints. Viewed analytically, SQP corresponds to applying Newton’s method to the KKT conditions. However, because linearized constraints only approximate the feasible region locally, the QP step can be overly optimistic—improving the objective while violating the true nonlinear constraints. Line search on an exact penalty function often forces very small step sizes, a limitation known as the Maratos effect~\cite{maratos1978exact}. Classical remedies, such as augmented Lagrangian merit functions~\cite{gill1986some} or second-order correction steps~\cite{mayne2009surperlinearly,byrd1987trust}, improve robustness but still do not guarantee anytime feasibility or monotonic descent. In model predictive control (MPC), SQP variants have been adapted to return feasible iterates, for instance through double-loop schemes~\cite{numerow2024inherently} or feasibility perturbation techniques~\cite{tenny2004nonlinear}. Nevertheless, these methods generally lack monotonic decrease guarantees and may still exhibit constraint violations during intermediate iterations.

\subsubsection{QCQP Methods}
Sequential quadratically constrained quadratic programming (SQCQP) methods follow a structure similar to SQP, but replace linearized constraints with quadratic approximations, leading to a QCQP subproblem at each iteration. Classical approaches construct second-order Taylor expansions of the constraints, which requires computing Hessians and generally yields nonconvex QCQPs. To avoid these difficulties, Fukushima~\cite{fukushima1986successive} proposed computing a search direction via a QCQP but, due to the lack of efficient QCQP solvers at the time, implemented this idea through two successive convex QPs. Variants for convex problems have also been explored~\cite{fukushima2003sequential}, though these methods still rely on line search with an $\ell_1$ penalty merit function and therefore do not guarantee anytime feasibility or monotonic descent.

A related line of work is the phase-I–phase-II SQCQP scheme of~\cite{west1992generalized}, which uses quadratic approximations of both the objective and the constraints sharing a common Hessian, chosen as a positive multiple of the identity. Once feasibility is achieved, this approach ensures both anytime feasibility and monotonic decrease, with convergence established under standard assumptions including twice differentiability, MFCQ, second-order sufficiency, and strict complementarity. However, the method requires including \emph{all} constraints in every QCQP subproblem, a significant limitation for large-scale problems with many constraints.

\subsubsection{First-Order Methods}
A growing body of work develops first-order algorithms for constrained optimization motivated by control theory~~\cite{allibhoy2023control,10.5555/3586589.3586845,raghunathan2025constrained} or Lipschitz-based surrogate modeling. One line of research designs continuous-time dynamics that guarantee invariance of the feasible set. In the safe gradient flow framework~\cite{allibhoy2021anytime}, each constraint is treated as a Control Barrier Function (CBF)~\cite{dawson2023safe}, yielding a vector field that follows the unconstrained gradient flow as closely as possible while enforcing feasibility in continuous time. However, when discretized, these flows generally lose their invariance properties--an issue closely related to the Maratos effect--so that anytime feasibility cannot be guaranteed in discrete time. A related formulation in~\cite{muehlebach2021constraints} provides a discrete-time update with convergence guarantees, but feasibility is not preserved at intermediate iterates and the stepsize depends on problem-dependent constants that are difficult to estimate in practice. \revise{Recent work on robust safe gradient flow for constrained reinforcement learning~\cite{mestres2025off} also adds a quadratic term to the CBF constraint to compensate for discretization and estimation errors, but its guarantees are developed for policy-gradient updates and rely on prescribed step sizes satisfying Lipschitz-dependent bounds or diminishing-schedule conditions rather than an adaptive line-search procedure.}

A complementary line of first-order methods constructs conservative local approximations of the objective and constraints using Lipschitz bounds. Representative examples include Successive Convex Approximation (SCA)~\cite{scutari2016parallel} and the Moving Ball Approximation (MBA) method~\cite{auslender2010moving}. While these approaches guarantee feasibility of the surrogate subproblem, they require knowledge of the Lipschitz constants of the gradients, which are often unavailable or hard to estimate. Lu et al.~\cite{lu2012sequential} proposed a line-search-based Sequential Convex Programming (SCP) algorithm that avoids explicit Lipschitz constants, but each line-search step involves solving a convex QCQP, resulting in substantially higher computational cost. Overall, existing first-order methods either lack anytime feasibility after discretization or rely on problem-dependent constants, limiting their applicability in real-time or large-scale settings.

\subsection{Notation}
We denote the space of $n$-dimensional real vectors by $\mathbb{R}^n$. For a vector $x \in \mathbb{R}^n$, $\|x\|_2$ denotes the Euclidean norm, and $\|x\|_P = \sqrt{x^\top P x}$ denotes the weighted $P$-norm, where $P \in \mathbb{R}^{n \times n}$ is symmetric positive definite ($P \succ 0$). For a natural number $m$, we use $[m] := \{1, 2, \ldots, m\}$ to denote the index set of the first $m$ integers. For $x \in \mathbb{R}^n$ and $r > 0$, $\mathcal{B}(x, r)$ denotes the open ball centered at $x$ with radius $r$.

\section{Background and Problem Statement}
This section introduces the constrained optimization problem studied in the paper, reviews the assumptions and optimality conditions used throughout, and summarizes key methods that motivate the development of our algorithm.

\subsection{Problem Setup and Assumptions}
We consider inequality-constrained, possibly nonconvex, optimization problems of the form
\begin{align}\label{eq:optimization_problem}
    \min_{x}~& f(x) \tag{OPT} \\
    \text{s.t.}~& g_i(x) \leq 0, \quad i \in [m], \notag
\end{align}
where $f:\mathbb{R}^n \rightarrow \mathbb{R}$ and $g_i:\mathbb{R}^n \rightarrow \mathbb{R}$ are continuously differentiable functions. We denote the feasible set as
\[
\mathcal{F} = \{x \in \mathbb{R}^n : g_i(x) \leq 0 ~\forall i \in [m]\}.
\] 
For $x \in \mathcal{F}$, we denote the set of active constraints at $x$ as $\mathcal{I}(x) = \{i \in [m] : g_i(x) = 0\}$.

\begin{definition}
    A continuously differentiable function $f:\mathbb{R}^n \rightarrow \mathbb{R}$ is $L$-smooth if its gradient is $L$-Lipschitz, i.e.,
    \begin{align*}
        \|\nabla f(x) - \nabla f(y)\|_2 \leq L\|x-y\|_2 ~\forall x,y.
    \end{align*}
\end{definition}
The $L$-smoothness of a function $f$ implies the following quadratic upper bound (see, e.g.,~\cite{nocedal2006numerical}):
\begin{align}\label{eq:l_smooth_upper_bound}
f(y) \le f(x) + \nabla f(x)^\top  (y - x) + \frac{L}{2}\|y - x\|_2^2 ~\forall x,y.
\end{align}

\begin{definition}[MFCQ]
    We say the Mangasarian-Fromovitz constraint qualification (MFCQ) holds at $x$ for \eqref{eq:optimization_problem} if $\exists d \in \mathbb{R}^n$ such that $\nabla g_i(x)^\top d < 0 ~ \forall i \in \mathcal{I}(x)$.
\end{definition}

\begin{assumption} \label{ass:assumption}
Throughout the paper, the following assumptions are imposed on~\eqref{eq:optimization_problem}.
\begin{enumerate} 
    \item [1.] $f$ is $L_f$-smooth and $g_i$ is $L_i$-smooth $\forall i \in [m]$.
    \item [2.] \eqref{eq:optimization_problem} is feasible and bounded below, and the optimal value is achieved at $x^\star$. 
    \item [3.] The Mangasarian-Fromovitz constraint qualification (MFCQ) holds for all $x \in \mathcal{F}$.
\end{enumerate}
\end{assumption}

\begin{definition}\label{def:stationary_point}
    A point $x \in \mathbb{R}^n$ is said to be \emph{stationary} for~\eqref{eq:optimization_problem} if there exists $\lambda \in \mathbb{R}^m$ such that the pair $(x,\lambda)$ satisfies the Karush–Kuhn–Tucker (KKT) conditions:
    \begin{subequations} \label{eq:KKT_conditions}
    \begin{align} 
        \nabla f(x) + \sum_{i=1}^m \lambda_i \nabla g_i(x) = 0, \label{eq:KKT_stationarity}\\
        g_i(x) \leq 0 ~ \forall i \in [m], \label{eq:KKT_primal_feasibility}\\
        \lambda \geq 0, \label{eq:KKT_dual_feasibility}\\
        \lambda_i g_i(x) = 0 ~ \forall i \in [m]. \label{eq:KKT_complementarity}
    \end{align}
    \end{subequations}
\end{definition}

\subsection{Background}

\subsubsection{Safe Gradient Flow} \label{sec:safe_gradient_flow}
Leveraging the theory of Control Barrier Functions (CBFs)~\cite{ames2016control}, 
Allibhoy \emph{et al.}~\cite{allibhoy2023control} introduced the \emph{safe gradient flow} dynamics for solving constrained optimization problems of the form~\eqref{eq:optimization_problem}. Specifically, the safe gradient flow is defined as
\begin{align}\label{eq:continuous_limit}
    \dot{x} = \tilde{u}(x),
\end{align}
where the velocity field $\tilde{u}(x)$ is obtained as the solution of the following convex quadratic program (QP):
\begin{align}\label{eq:direction_qp}
    \tilde{u}(x) := \arg\min_{u}~& \frac{1}{2}\|u + \nabla f(x)\|_2^2 \\
    \text{s.t.}~& \nabla g_i(x)^\top u \le -\alpha g_i(x), \quad \forall i \in [m], \notag
\end{align}
with $\alpha > 0$ a prescribed parameter. This QP can be interpreted in two complementary ways.

From a \emph{control-theoretic} perspective, the QP computes the smallest correction to the unconstrained gradient descent direction that ensures satisfaction of the differential inequalities
\begin{align}\label{eq:invariance_condition}
    \dot{g}_i(x) + \alpha g_i(x) \le 0, \quad \forall i \in [m],
\end{align}
which serve as control barrier conditions. These inequalities guarantee both \emph{forward invariance} and \emph{exponential stability} of the feasible set. 
In particular, if the initial condition is feasible, $g_i(x(0)) \le 0$, then $g_i(x(t)) \le 0$ for all $t \ge 0$; 
moreover, if a constraint is initially violated, $g_i(x(0)) > 0$, its violation decays exponentially, i.e.,
\[
g_i(x(t)) \le g_i(x(0)) e^{-\alpha t}.
\]
From an \emph{optimization} standpoint, when $\alpha = 1$, the QP in~\eqref{eq:direction_qp} projects the negative gradient $-\nabla f(x)$ onto the linearized feasible set
\begin{align}\label{eq:projection_interpretation}
    \{\, u : \nabla g_i(x)^\top  u \le - g_i(x),~ \forall i \in [m] \,\},
\end{align}
This coincides with the SQP search direction under an identity Hessian approximation of the Lagrangian and has also appeared in the MPC literature as a means of generating feasible descent directions~\cite{torrisi2018projected}.

Under the Mangasarian–Fromovitz constraint qualification (MFCQ) and the constant rank condition (CRC)~\cite{liu1995sensitivity}, the mapping $x \mapsto \tilde{u}(x)$ is locally Lipschitz~\cite{allibhoy2023control}.
This guarantees existence and uniqueness of solutions of~\eqref{eq:continuous_limit}.
Moreover, if $x^\star$ satisfies $\tilde{u}(x^\star)=0$, then $(x^\star,\lambda^\star)$ satisfies the KKT conditions of~\eqref{eq:optimization_problem}, and conversely.
If $x^\star$ is an isolated stationary point, it is asymptotically stable; exponential stability follows under LICQ, strict complementarity, and the second-order sufficient condition~\cite[Thm.~5.7]{allibhoy2023control}.

Although these results provide strong guarantees for the continuous-time system, the analysis in~\cite{allibhoy2023control} does not address the discretization required for numerical implementation.
In particular, the anytime feasibility property—which holds exactly in continuous time—may be lost upon discretization.

\subsubsection{Constrained Gradient Flow (Descent)} \label{sec:constrained_gradient_flow}
In closely related work, Muehlebach \emph{et al.}~\cite{muehlebach2021constraints} proposed the \emph{constrained gradient flow}, in which the invariance condition~\eqref{eq:invariance_condition} is enforced only for active or violated constraints.
This modification reduces the number of constraints appearing in the QP~\eqref{eq:direction_qp}, but it renders the mapping $x \mapsto \tilde{u}(x)$ nonsmooth.
The authors further developed a constant–step size Euler discretization of the resulting dynamics and established that, under convexity assumptions on problem~\eqref{eq:optimization_problem}, the squared distance to an optimal solution decreases at a rate of $\mathcal{O}(1/k)$. However, the discrete-time scheme in~\cite{muehlebach2021constraints} does not guarantee feasibility of the iterates, and its admissible step size depends on problem-dependent constants that are often difficult to estimate in practice.

In the next section, we address these issues by developing a new continuous-time dynamical system and an associated safeguarded discretization procedure that provably preserves feasibility at every iteration while ensuring monotonic descent of the objective.


\section{Safe Sequential QCQP}\label{sec:sequential_qcqp}
In this section, we develop a continuous-time dynamical system $\dot{x}=u(x)$ and an associated Euler discretization for solving problem~\eqref{eq:optimization_problem}.
Given an iterate $x^{(k)}$, the update takes the form
\begin{align}\label{eq:update_rule}
x^{(k+1)} = x^{(k)} + t^{(k)} u(x^{(k)}),
\end{align}
where the vector field (or search direction) $u(x^{(k)})$ and step size $t^{(k)}>0$ are designed to guarantee two fundamental properties: \textit{(1) Anytime feasibility:}  
If the initial iterate $x^{(0)}$ satisfies $g_i(x^{(0)}) \le 0$ for all $i \in [m]$, then feasibility is maintained at every iteration, i.e.,
\begin{align}\label{eq:anytime_feasibility}
    g_i(x^{(k)}) \le 0, \quad \forall k \ge 0,~ i \in [m].
\end{align}
\textit{(2) Monotonic descent:}  
The sequence of objective values is strictly decreasing,
\begin{align}\label{eq:monotonic_descent}
    f(x^{(k+1)}) < f(x^{(k)}), \quad \forall k \ge 0.
\end{align}

The proposed method, referred to as the \emph{Safe Sequential QCQP (SS-QCQP) algorithm}, constructs $u(x^{(k)})$ by solving a convex quadratically constrained quadratic program (QCQP) that locally approximates~\eqref{eq:optimization_problem} while enforcing feasibility (\Cref{sec:construction_of_search_direction}). 
The step size $t_k$ is then selected to guarantee both a sufficient decrease in the objective and preservation of safety along the resulting feasible direction (\Cref{subsec:step_size_selection}).

\subsection{ODE Construction} \label{sec:construction_of_search_direction}
Discretizing the safe gradient flow~\eqref{eq:continuous_limit} into an iterative algorithm introduces a fundamental difficulty: the linearization of the constraints in~\eqref{eq:direction_qp} neglects curvature, which can cause infeasibility for any positive step size. To illustrate this, consider a point $x$ where $g_i(x) = 0$ for some $i \in [m]$ (an active constraint). 
In this case, the feasible set of~\eqref{eq:direction_qp} permits
\begin{align}\label{eq:tangent_condition}
    \nabla g_i(x)^\top  \tilde{u}(x) = 0,
\end{align}
meaning that $\tilde{u}(x)$ can be tangent to the zero-level set of $g_i$. 
Consequently, taking any positive step in the direction of $\tilde{u}(x)$ may violate feasibility, i.e.,
\begin{align}\label{eq:maratos_infeasibility}
    g_i(x + t \tilde{u}(x)) > 0, \quad \forall\, t > 0.
\end{align}
This phenomenon, similar to the \emph{Maratos effect}, arises from the mismatch between the local linear approximation of the constraints and their true nonlinear geometry.

To mitigate this limitation, we incorporate a quadratic correction term into the feasibility constraints of the QP~\eqref{eq:direction_qp}, leading to the following convex quadratically constrained quadratic program (QCQP):

\begin{align}\label{eq:direction_qcqp} \tag{\ssqcqp{$x$}}
    u(x) := \arg\min_{u}~& \frac{1}{2}\|u + \nabla f(x)\|_2^2 \\
    \text{s.t.}~& \nabla g_i(x)^\top  u \le -\alpha g_i(x) - w_i\|u\|_2^2,
    \  \forall i \in [m], \notag
\end{align}
where $\alpha > 0$ and $w_i > 0$ are prescribed parameters. 
When $g_i(x) = 0$, i.e., when $x$ lies on the boundary of the feasible region, the additional quadratic term $w_i\|u\|_2^2$ ensures that
\begin{align}\label{eq:strict_tangent_negativity}
    \nabla g_i(x)^\top  u(x) < 0,
\end{align}
thereby preventing the vector field from being tangent to the constraint boundary.  Geometrically, the penalty $w_i|u|^2$ tilts the direction towards the interior of the feasible set (see~\Cref{fig:directions}), thereby providing a curvature-aware safeguard absent in the original QP formulation.

In the next subsection, we establish key properties of the direction field $u(x)$ defined by~\eqref{eq:direction_qcqp}, including feasibility preservation, regularity, and its correspondence with KKT points of~\eqref{eq:optimization_problem}.


\begin{figure} 
    \centering
    \includegraphics[width=0.7\linewidth, trim=0 0 0 0]{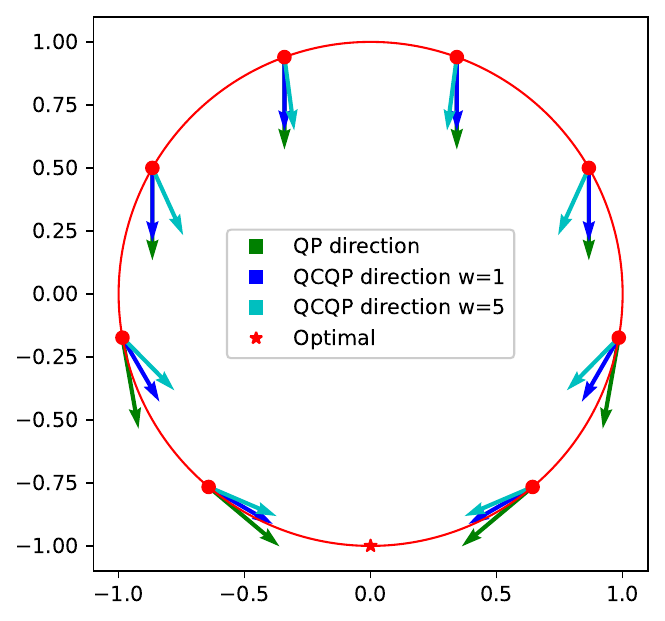}
    \caption{The search directions computed by solving the QP \eqref{eq:direction_qp} and the QCQP \eqref{eq:direction_qcqp} where $x = \begin{bmatrix}x_1, x_2\end{bmatrix}^\top $, the objective $f(x)=x_2$ is a linear function and the constraint $g(x)=\|x\|_2^2-1 \leq 0$ is a unit ball constraint. Here we scale the search directions differently so that they do not overlap. When $x$ is on the boundary of the feasible region, i.e., $g(x)=0$, \eqref{eq:direction_qp} might generate a direction tangent to the boundary, while \eqref{eq:direction_qcqp} tilts the direction towards the interior of the feasible region.}
    \label{fig:directions}
\end{figure}

\subsection{Properties of the Search Direction} \label{sec:properties_of_search_direction}
We begin by showing that~\eqref{eq:direction_qcqp} always admits a strictly feasible point whenever the original problem~\eqref{eq:optimization_problem} satisfies MFCQ. Because the QCQP is \revise{strongly} convex, strict feasibility guarantees that the optimizer $u(x)$ exists and is unique.

\begin{lemma}[Strict Feasibility of~\eqref{eq:direction_qcqp}]
\label{lem:qcqp_strictly_feasible}
Suppose the MFCQ condition holds for~\eqref{eq:optimization_problem} and that $x$ is feasible. 
Then ~\eqref{eq:direction_qcqp} is strictly feasible.
\end{lemma}

\begin{proof}
Under MFCQ, there exists a vector \revise{$d$} such that
\revise{
\[
\nabla g_i(x)^\top  d < 0, \quad \forall i \in \mathcal{I}(x),
\]
}
where $\mathcal{I}(x) := \{i \in [m] : g_i(x) = 0\}$ denotes the set of active constraints. We consider two cases.

\smallskip
\noindent\textbf{Case 1:} $i \in \mathcal{I}(x)$ (active constraints).  
Since $g_i(x) = 0$ and \revise{$\nabla g_i(x)^\top  d < 0$}, choose any
\[
0 < \eta_1 < \min_{i \in \mathcal{I}(x)} 
    \frac{-\nabla g_i(x)^\top  \revise{d}}{w_i \|\revise{d}\|_2^2}.
\]
Then \revise{$\hat{d}_1 := \eta_1 \revise{d}$} satisfies
\[
\nabla g_i(x)^\top  \revise{\hat{d}_1 }
\!=\! \eta_1 \nabla g_i(x)^\top  \revise{d}
\!<\! - w_i \|\revise{\hat{d}_1}\|_2^2
\!=\! - \alpha g_i(x) \!-\! w_i \|\revise{\hat{d}_1}\|_2^2,
\]
for all $i \in \mathcal{I}(x)$.

\smallskip
\noindent\textbf{Case 2:} $i \notin \mathcal{I}(x)$ (inactive constraints).  
For these indices, $g_i(x) < 0$. The inequality
\[
\nabla g_i(x)^\top  (\eta \revise{d})
\le -\alpha g_i(x) - w_i \|\eta \revise{d}\|_2^2
\]
is quadratic in $\eta$. It holds strictly for any
\[
0 < \eta_2 < 
\min_{i \notin \mathcal{I}(x)}
\tfrac{-\nabla g_i(x)^\top  \revise{d} + 
\sqrt{(\nabla g_i(x)^\top  \revise{d})^2 
- 4 \alpha w_i \|\revise{d}\|_2^2 g_i(x)}}
{2 w_i \|\revise{d}\|_2^2}.
\]

\smallskip
\noindent Finally, let $\eta := \min\{\eta_1, \eta_2\}$ and set $\revise{\hat{d}} := \eta \revise{d}$.  
Then $\hat{u}$ satisfies
\[
\nabla g_i(x)^\top  \revise{\hat{d}}
< -\alpha g_i(x) - w_i \|\revise{\hat{d}}\|_2^2,
\quad \forall i \in [m],
\]
showing that~\eqref{eq:direction_qcqp} is strictly feasible.
\end{proof}

\begin{theorem}\label{thm:direction_qcqp}
Assume that problem~\eqref{eq:optimization_problem} satisfies~\Cref{ass:assumption}, and let $x$ be a feasible point of~\eqref{eq:optimization_problem}. Then the solution $u(x)$ to~\eqref{eq:direction_qcqp} satisfies:
\begin{enumerate}
    \item $u(x) = 0$ if and only if $x$ is a stationary point of~\eqref{eq:optimization_problem}. 
    \item The search direction is a descent direction for $f$, i.e.,
    \begin{align}\label{eq:descent_property}
        \nabla f(x)^\top  u(x) \le -\|u(x)\|_2^2.
    \end{align}
\end{enumerate}
\end{theorem}

\begin{proof}
Since~\eqref{eq:direction_qcqp} is convex and strictly feasible (by the preceding lemma), strong duality holds and the KKT conditions are necessary and sufficient~\cite{boyd2004convex}.
Define
\[
c_i(u;x) \;:=\; \nabla g_i(x)^\top  u \;+\; w_i\|u\|_2^2 \;+\; \alpha g_i(x),\quad i\in[m].
\]
The KKT conditions of~\eqref{eq:direction_qcqp} state that there exists $\lambda(x)\in\mathbb{R}^m_{\ge 0}$ such that
\begin{subequations}\label{eq:QCQP_KKT}
\begin{align}
    &u(x) + \nabla f(x) + \sum_{i=1}^m \lambda_i(x)\big(\nabla g_i(x) + 2w_i u(x)\big) = 0, \label{eq:QCQP_KKT_stationarity}\\
    &c_i\big(u(x);x\big) \le 0,\quad i\in[m], \label{eq:QCQP_KKT_primal}\\
    &\lambda_i(x) \ge 0,\quad i\in[m], \label{eq:QCQP_KKT_dual}\\
    &\lambda_i(x)\,c_i\big(u(x);x\big) = 0,\quad i\in[m]. \label{eq:QCQP_KKT_comp}
\end{align}
\end{subequations}

\noindent\textit{Part (1):}
($\implies$) Setting $u(x)=0$ in~\eqref{eq:QCQP_KKT_stationarity} gives
\[
\nabla f(x) + \sum_{i=1}^m \lambda_i(x)\,\nabla g_i(x) = 0.
\]
Complementarity~\eqref{eq:QCQP_KKT_comp} with $u(x)=0$ yields
\(
\alpha \lambda_i(x) g_i(x) = 0
\),
so $\lambda_i(x)=0$ whenever $g_i(x)<0$, while $\lambda_i(x)\ge 0$ is allowed if $g_i(x)=0$.
Thus, $x$ together with multipliers satisfies the KKT conditions of~\eqref{eq:optimization_problem}, proving stationarity.

($\impliedby$) Let $x$ be a stationary point of \eqref{eq:optimization_problem}. By \Cref{def:stationary_point}, there exists $\lambda$ such that $(x,\lambda)$ satisfies \eqref{eq:KKT_conditions}. Letting $u=0$, \eqref{eq:KKT_stationarity} gives 
\[
u + \nabla f(x) + \sum_{i=1}^m \lambda_i(x)\big(\nabla g_i(x) + 2w_i u\big) = 0.
\]
Complementarity \eqref{eq:KKT_complementarity} yields $\lambda_i c_i(0;x) = 0$, and \eqref{eq:KKT_primal_feasibility} yields $c_i(0, ;x) = \alpha g_i(x) \leq 0$. Thus, the pair $(0, \lambda)$ satisfies the KKT conditions \eqref{eq:QCQP_KKT}. Since \eqref{eq:direction_qcqp} is a convex problem with a strongly convex objective, $u=0$ is the unique solution to \eqref{eq:direction_qcqp}.

\smallskip
\noindent\textit{Part (2):}
Taking the inner product of~\eqref{eq:QCQP_KKT_stationarity} with $u(x)$ gives
\begin{small}
    \begin{align*}
         &\nabla f(x)^\top  u(x) \\
        =& -\big(1 + 2\sum_{i=1}^m \lambda_i(x) w_i \big) \|u(x)\|_2^2 - \sum_{i=1}^m \lambda_i(x) \nabla g_i(x)^\top  u(x) \\
        =& -\big(1 + 2\sum_{i=1}^m \lambda_i(x) w_i\big) \|u(x)\|_2^2 + \sum_{i=1}^m \lambda_i(x) \big( \alpha g_i(x) + w_i \|u(x)\|_2^2\big) \\
        =& -\big(1 + \sum_{i=1}^m \lambda_i(x) w_i\big) \|u(x)\|_2^2 + \sum_{i=1}^m \lambda_i(x) \alpha g_i(x) \\
        \leq& - \|u(x)\|_2^2.
    \end{align*}
    \end{small}
    Here we use the complementarity \eqref{eq:QCQP_KKT_comp} to get the second equality.
\end{proof}

\subsection{Convergence Analysis of Continuous-Time Dynamics}
With the properties of the search direction $u(x)$ established in~\Cref{thm:direction_qcqp}, we define the continuous-time system 
\begin{align}\label{eq:ode_qcqp}
    \dot{x}(t)=u(x(t)) \quad \text{with } x(0) \text{ given}
\end{align}
to solve~\eqref{eq:optimization_problem}, where $u(x)$ is the unique solution to ~\eqref{eq:direction_qcqp}. For this ODE to be well posed, we must ensure that the mapping $x \mapsto u(x)$ is continuous on a neighborhood of the feasible set~$\mathcal F$.
This property guarantees the existence of solutions to the ODE. In the following lemma, we will first extend the strict feasibility result of Lemma~\ref{lem:qcqp_strictly_feasible} from feasible points to an entire open neighborhood of~$\mathcal F$. Then, we will show that every solution trajectory converges to the set of first-order stationary points of~\eqref{eq:optimization_problem}.

\begin{lemma}\label{lem:qcqp_strictly_feasible_open_set}
There exists an open set $U$ containing $\mathcal{F}$ such that the QCQP~\eqref{eq:direction_qcqp} is strictly feasible for all $x \in U$.
\end{lemma}

\begin{proof}
For any $x \in \mathcal{F}$, \Cref{lem:qcqp_strictly_feasible} establishes that~\eqref{eq:direction_qcqp} is strictly feasible; that is, there exist $u$ and $\delta_i < 0$ such that  
\[
    \alpha g_i(x) + \nabla g_i(x)^\top u + w_i\|u\|_2^2 = \delta_i, 
    \quad \forall i \in [m].
\]
Since each $g_i$ is $L_i$-smooth, there exists $\epsilon_i(x) > 0$ such that  
\[
    \alpha g_i(y) + \nabla g_i(y)^\top u + w_i\|u\|_2^2 
    \le \tfrac{\delta_i}{2} < 0,
\]
for all $y \in \mathcal{B}(x, \epsilon_i(x))$.  
Hence, (\ssqcqp{$y$}) remains strictly feasible for all 
$y \in \mathcal{B}\big(x, \min_i \epsilon_i(x)\big)$.  
Defining 
\[
    U := \bigcup_{x \in \mathcal{F}} \mathcal{B}\big(x, \min_i \epsilon_i(x)\big),
\]
we conclude that~\eqref{eq:direction_qcqp} is strictly feasible for all $x \in U$.
\end{proof}

\revise{For each $x\in U$, the QCQP is strictly feasible by Lemma~2 and has a unique optimizer because its objective is strongly convex and coercive. Moreover, the objective and constraint functions depend continuously on $x$. Strict feasibility ensures that a strictly feasible direction at any given parameter remains feasible for all sufficiently nearby parameters. Together with the coercive quadratic objective, this provides local boundedness of the optimizers and, by standard parametric optimization arguments, continuity of the solution mapping $x\mapsto u(x)$ on $U$ \cite{bonnans2013perturbation}.} 

\subsubsection*{Existence of Trajectories} 
The continuity of $u(\cdot)$ established above guarantees existence of solutions to the ODE~\eqref{eq:ode_qcqp}, but it does not ensure uniqueness.  
Uniqueness would require local Lipschitz continuity of $u(\cdot)$, which in turn depends on the stability of the active-set structure of~\eqref{eq:direction_qcqp}. Such stability is typically ensured by the \emph{Constant Rank Constraint Qualification} (CRC) applied directly to the QCQP.  
However, CRC is considerably stronger than MFCQ and is rarely imposed in nonlinear programming, where regularity assumptions are made for the \emph{original} problem~\eqref{eq:optimization_problem} rather than for subproblems such as~\eqref{eq:direction_qcqp}.

Moreover, CRC for~\eqref{eq:optimization_problem} does not imply CRC for~\eqref{eq:direction_qcqp}.  
If the original problem satisfies CRC at a feasible point $x$, then the gradients of its active constraints,  
$\{\nabla g_i(x): g_i(x)=0\}$,  
are linearly independent.  
But when $x$ is strictly feasible, this set is empty, and CRC provides no information.  
In contrast, the QCQP~\eqref{eq:direction_qcqp} may still have active constraints, 
\[
\nabla g_i(x)^\top u(x) + \alpha g_i(x) + w_i\|u(x)\|_2^2 = 0.
\]
The gradients associated with these QCQP-active constraints,
\[
\big\{\, \nabla g_i(x) + 2 w_i u(x) :
\nabla g_i(x)^\top u(x) + \alpha g_i(x) + w_i\|u(x)\|_2^2 = 0 \,\big\},
\]
need not be linearly independent, even if the original problem satisfies CRC.  
Thus, the QCQP subproblem may fail to satisfy CRC regardless of any regularity properties possessed by~\eqref{eq:optimization_problem}.

For these reasons, we do not assume uniqueness of trajectories, and all convergence guarantees below apply to \emph{every} solution of the ODE.

\begin{theorem}
Suppose problem~\eqref{eq:optimization_problem} satisfies~\Cref{ass:assumption}, and $x(t)$ is any solution to the ODE \eqref{eq:ode_qcqp} with a  feasible initial condition $x(0) \in \mathcal{F}$. 
Then $x(t)$ satisfies the following properties:
\begin{enumerate}
    \item \textit{Anytime feasibility:}
    \[
        g_i(x(t)) \le 0, \quad \forall\, t > 0,~ i \in [m].
    \]
    \item \textit{\revise{Ergodic} convergence:}
    \[
        \tfrac{1}{2}\int_0^t \|u(x(\tau))\|_2^2 \, d\tau 
        \le f(x(0)) - f^\star, 
        \quad \forall\, t > 0.
    \]
    where $f^\star$ is the optimal value of \eqref{eq:optimization_problem}.
\end{enumerate}
\end{theorem}
   
\begin{proof}
    $\forall i\in [m]$, when $g_i(x)=0$, we have
    \[
    \dot{g}_i(x(t)) = \nabla g_i(x)^\top u(x(t)) \leq -\alpha g_i(x) - w_i\|u(x(t))\|_2^2 \leq 0.
    \]
    According to Nagumo’s theorem~\cite{blanchini2008set}\cite{blanchini1999set}, the 0-sublevel set $\{x : g_i(x) \leq 0\}$ is forward-invariant. 
    
    Let $f^\star$ be the optimal value of  \eqref{eq:optimization_problem} and define the Lyapunov function $V(x(t)) = f(x(t)) - f^\star + \frac{1}{2} \int_0^t \|u(x(\tau))\|_2^2 d\tau$. Using the result $\nabla f(x(t))^\top u(x(t)) \leq -\|u(x(t))\|_2^2$ from \Cref{thm:direction_qcqp}, we have 
    \begin{align*}
        \dot{V}(x(t)) 
        &= \nabla f(x(t))^\top u(x(t)) + \frac{1}{2}\|u(x(t))\|_2^2  \\
        &\leq  -\frac{1}{2}\|u(x(t))\|_2^2
    \end{align*}
    Thus $V(x(t)) \leq V(x(0))~\forall t>0$. Since $f(x(t)) - f^\star \geq 0$, we have
    \begin{align*}
        \frac{1}{2} \int_0^t \|u(x(\tau))\|_2^2 d\tau \leq V(x(t)) \leq V(x(0)) = f(x(0)) - f^\star.
    \end{align*}
    This implies that 
    \begin{align*}
        \min_{\tau \in [0,t]} \|u(x(\tau))\|_2^2 \leq \frac{1}{t} \int_0^t \|u(x(\tau))\|_2^2 d\tau \leq \frac{2(f(x(0)) - f^\star)}{t},
    \end{align*}
\end{proof}

\begin{remark}[Non-ergodic vs.\ ergodic rates and stationarity]
The estimate in the theorem yields the \revise{best-iterate \emph{non-ergodic}} rate
\[
\min_{\tau\in[0,t]}\|u(x(\tau))\|_2^2 \;\le\; \frac{2\big(f(x(0))-f^\star\big)}{t}.
\]
By averaging, we also obtain the \revise{\emph{ergodic}} bound
\[
\frac{1}{t}\int_{0}^{t}\|u(x(\tau))\|_2^2\,d\tau
\;\le\; \frac{2\big(f(x(0))-f^\star\big)}{t}.
\]
Consequently $\int_0^\infty \|u(x(t))\|_2^2 dt < \infty$, so there exists a sequence $t_k\!\to\!\infty$ with $\|u(x(t_k))\|\!\to\!0$. 
\end{remark}

\subsection{Safeguarded Adaptive Discretization}
\label{subsec:step_size_selection}
The proposed formulation~\eqref{eq:direction_qcqp} enables \emph{automatic step size selection} via a safeguarded Armijo-type line search that guarantees both feasibility and descent at every iteration. 
Starting from a feasible point $x_0$, we discretize the ODE $\dot{x}(t) = u(x(t))$ using the forward Euler method, where $u(x_k)$ is obtained by solving~\eqref{eq:direction_qcqp} and $t_k > 0$ is chosen via a backtracking line search so that the next iterate
\[
x^{(k+1)} = x^{(k)}+t^{(k)} u(x^{(k)})
\]
satisfies
\hspace{-2mm}
\begin{align}\label{eq:backtracking_line_search}
    &f(x^{(k+1)}) \le f(x^{(k)}) + \gamma \nabla f(x^{(k)})^\top (x^{(k+1)} - x^{(k)}), \\
    &g_i(x^{(k+1)}) \le 0, 
    \quad \forall i \in [m], \notag
\end{align}
where $\gamma \in (0,1)$ is fixed. The first inequality in~\eqref{eq:backtracking_line_search} is the classical Armijo condition, ensuring sufficient descent in the objective.  
The rest guarantees \emph{anytime feasibility}, guaranteeing that all iterates remain in the feasible set of~\eqref{eq:optimization_problem}.

The next lemma shows that, under~\Cref{ass:assumption}, the backtracking procedure in~\eqref{eq:backtracking_line_search} terminates in finitely many steps and returns a step size bounded away from zero.


\begin{lemma}\label{lem:qcqp_step_size}
Suppose $x_k$ is feasible for~\eqref{eq:optimization_problem}, and that $w_i \ge \underline{w}$ for all $i \in [m]$, for some $\underline{w} > 0$.  
Then, the backtracking conditions~\eqref{eq:backtracking_line_search} are satisfied for all 
$t \in [0, \underline{t}]$, where
\begin{align}\label{eq:step_lower_bound}
    \underline{t} = \min\bigg\{
        \frac{1}{\alpha},
        \frac{2(1-\gamma)}{L_f},
        \frac{2\underline{w}}{L_1}, \ldots,
        \frac{2\underline{w}}{L_m}
    \bigg\}.
\end{align}
\end{lemma}
\begin{proof}
We first establish feasibility.  
For any $t \in [0, \underline{t}]$, it follows that 
$t \le \min\{\tfrac{1}{\alpha}, \tfrac{2w_i}{L_i}\}$ for all $i = 1,\ldots,m$.  
Using the $L_i$-smoothness of $g_i$ and the feasibility constraint in~\eqref{eq:direction_qcqp}, we obtain
\begin{align*}
    &g_i(x + t u(x))
    \le g_i(x) + t\,\nabla g_i(x)^\top  u(x)
       + \tfrac{L_i t^2}{2}\|u(x)\|_2^2 \\
    &\le g_i(x) + t[-\alpha g_i(x) - w_i\|u(x)\|_2^2]
       + \tfrac{L_i t^2}{2}\|u(x)\|_2^2 \\
    &= (1 - \alpha t) g_i(x)
       + \Big(\tfrac{L_i t^2}{2} - w_i t\Big)\|u(x)\|_2^2 \\
    &\le (1 - \alpha t) g_i(x) \le 0,
\end{align*}
which establishes that $x + t u(x)$ is feasible for all $t \in [0, \underline{t}]$.

We now verify the Armijo condition.  
By the $L_f$-smoothness of $f$, we have
\begin{align*}
    f(x + t u(x))
    &\le f(x) + t\,\nabla f(x)^\top  u(x)
       + \tfrac{L_f t^2}{2}\|u(x)\|_2^2.
\end{align*}
From \Cref{thm:direction_qcqp}, the search direction satisfies 
$\nabla f(x)^\top  u(x) \le -\|u(x)\|_2^2$.  
Substituting this bound yields
\begin{align*}
    f(x + t u(x))
    &\le f(x) + t\,\nabla f(x)^\top  u(x)
       - \tfrac{L_f t^2}{2}\,\nabla f(x)^\top  u(x) \\
    &= f(x) + \Big(\tfrac{L_f t^2}{2} - t\Big)\big(-\nabla f(x)^\top  u(x)\big).
\end{align*}
For $t \le \tfrac{2(1-\gamma)}{L_f}$, the coefficient 
$(\tfrac{L_f t^2}{2} - t)$ is upper-bounded by \revise{$-\gamma t$}, giving
\begin{align*}
    f(x + t u(x)) 
    \le f(x) + \gamma t\,\nabla f(x)^\top  u(x),
\end{align*}
which confirms the Armijo condition.  
Thus, both conditions in~\eqref{eq:backtracking_line_search} hold for all $t \in [0, \underline{t}]$.
\end{proof}
%
%
%
%
%

In \Cref{lem:qcqp_step_size}, we only require that the parameters $w_i$ are lower bounded by some $\underline{w} > 0$.  
This condition guarantees a uniformly positive lower bound on the step size, unlike the pure QP case, where the step size can vanish, thus preventing infeasibility and stagnation during backtracking line search. 
At the same time, it allows for the adaptive selection of $w_i$, as discussed in  \Cref{sec:w_choice}, enabling the algorithm to adjust the strength of the quadratic correction dynamically while preserving the theoretical convergence guarantees. The overall algorithm is summarized in \Cref{alg:SQCQP}.

\begin{algorithm}[H]
\caption{Safe Sequential QCQP (SS-QCQP) for solving~\eqref{eq:optimization_problem}}
\label{alg:SQCQP}
\begin{algorithmic}[1]
    \State \textbf{Input:} initial feasible point $x^{(0)}$ such that $g_i(x^{(0)}) \le 0$ for all $i\in [m]$; 
    parameters $\gamma \in (0,1)$, \revise{$\beta \in (0,1)$}, $\underline{w} > 0$, $\alpha > 0$, and tolerance $\epsilon > 0$.
    \For{$k = 0,1,2,\ldots$}
        \State Choose adaptive parameters $\{\,w_i^{(k)} \ge \underline{w} : i\in [m]\}$.
        \State Compute the search direction $u = u(x^{(k)})$ by solving~\eqref{eq:direction_qcqp} with $w_i = w_i^{(k)}$.
        \If{$\|u\|_2 \le \epsilon$}
            \State \Return $x^{(k)}$ \Comment{Convergence achieved}
        \EndIf
        \State Initialize step size $t^{(k)} = 1$.
        \While{
            $\exists i \in [m]$ such that 
            $g_i(x^{(k)} + t^{(k)} u) > 0$
            \textbf{or}
            $f(x^{(k)} + t^{(k)} u) > f(x^{(k)}) + \gamma t^{(k)} \nabla f(x^{(k)})^\top  u$
        }
            \State \revise{$t^{(k)} := \beta t^{(k)}$ }\Comment{Backtracking line search}
        \EndWhile
        \State Update $x^{(k+1)} := x^{(k)} + t^{(k)} u$.
    \EndFor
\end{algorithmic}
\end{algorithm}

In the next result, we show that the convergence rate of the discrete-time algorithm matches that of its continuous-time counterpart.

\begin{theorem}[Convergence Guarantee of SS-QCQP]\label{thm:convergence_sqcqp}
Let ~\Cref{ass:assumption} hold for ~\eqref{eq:optimization_problem}.  
Then, for all iterations $k = 0,1,\ldots$ generated by ~\Cref{alg:SQCQP}, the following bound holds:
\begin{align}\label{eq:sqcqp_convergence_bound}
    \min_{i = 0,\ldots,k} \|u(x^{(i)})\|_2^2
    \le \frac{f(x^{(0)}) - f^\star}{\gamma\,\hat{t}\,(k+1)},
\end{align}
where \revise{$\hat{t} = \min\{1, \beta \underline{t}\}$}, and $\underline{t}$ is the uniform lower bound on the step size established in \Cref{lem:qcqp_step_size}.
\end{theorem}
\begin{proof}
\revise{
Using \Cref{lem:qcqp_step_size}, at each iteration $k$, the line search criterion is satisfied for any $t^{(k)} \in [0, \underline{t}]$. Starting from $t^{(k)}=1$, since the line search generates $t^{(k)} =\{1, \frac{1}{\beta}, \frac{1}{\beta^2}, ...\}$, the resulting step size satisfies $t^{(k)} \geq \min\{1, \beta \underline{t}\}$.}
By the Armijo condition in~\eqref{eq:backtracking_line_search} and the descent property 
$\nabla f(x^{(k)})^\top  u(x^{(k)}) \le -\|u(x^{(k)})\|_2^2$ from \Cref{thm:direction_qcqp}, we obtain
\begin{align*}
    f(x^{(k+1)}) 
    &\le f(x^{(k)}) + \gamma t^{(k)} \nabla f(x^{(k)})^\top  u(x^{(k)}) \\
    &\le f(x^{(k)}) - \gamma t^{(k)} \|u(x^{(k)})\|_2^2 \\
    &\le f(x^{(k)}) - \gamma \hat{t} \|u(x^{(k)})\|_2^2.
\end{align*}
Summing over $i = 0,\ldots,k$ and telescoping the inequality yields
\begin{align*}
    \sum_{i=0}^{k} \gamma \hat{t} \|u(x^{(i)})\|_2^2
    \le f(x^{(0)}) - f(x^{(k+1)})
    \le f(x^{(0)}) - f^\star.
\end{align*}
Dividing both sides by $\gamma \underline{t} (k+1)$ gives
\begin{align*}
    \min_{i=0,\ldots,k} \|u(x^{(i)})\|_2^2
    &\le \frac{1}{k+1} \sum_{i=0}^{k} \|u(x^{(i)})\|_2^2 \\
    &\le \frac{f(x^{(0)}) - f^\star}{\gamma \hat{t} (k+1)},
\end{align*}
which establishes the result.
\end{proof}


\revise{
The bound in~\eqref{eq:sqcqp_convergence_bound} implies
an $O(1/k)$ best-iterate convergence rate for the stationarity measure $\|u(x)\|_2^2$,
}
which directly parallels the continuous-time result in \Cref{thm:convergence_sqcqp}.  
By averaging over the iterates, one also obtains the corresponding \emph{ergodic} bound
\[
    \frac{1}{k+1} \sum_{i=0}^{k} \|u(x^{(i)})\|_2^2
    \le \frac{f(x^{(0)}) - f^\star}{\gamma\,\underline{t}\,(k+1)},
\]
Consequently, $\sum_{k=0}^{\infty}\|u(x^{(k)})\|_2^2 < \infty$, which implies $\|u(x^{(k)})\|\!\to\!0$. 

\begin{remark}[Comparison with the MBA method~\cite{auslender2010moving}]
The design of the proposed SS-QCQP algorithm is conceptually inspired by the Moving Balls Approximation (MBA) method~\cite{auslender2010moving}, which constructs locally feasible quadratic models for nonlinear constrained problems.  
MBA iteratively solves
\begin{align}\label{eq:MBA}
    \min_{y}~& f(x) + \nabla f(x)^\top (y - x) + \tfrac{L_f}{2}\|y - x\|_2^2 \\
    \text{s.t.}~& g_i(x) + \nabla g_i(x)^\top (y - x) + \tfrac{L_i}{2}\|y - x\|_2^2 \le 0, 
    \quad \forall i \in [m]. \notag
\end{align}
The key distinction lies in how feasibility is enforced.  
MBA builds a conservative local approximation using Lipschitz upper bounds (see \eqref{eq:l_smooth_upper_bound}) so that the solution to~\eqref{eq:MBA} is guaranteed to be feasible and can be directly adopted as the next iterate.  
In contrast, ~\eqref{eq:direction_qcqp} computes a feasible \emph{search direction} and updates the iterate through a safeguarded line search.  Thus, unlike MBA, the SS-QCQP framework does not require prior knowledge of the Lipschitz constants $L_f$ and $L_i$, which are often difficult to estimate in practice.  We note that extensions of MBA incorporating line search have been proposed to relax this requirement~\cite{lu2012sequential}, but each line-search step requires solving a convex QCQP, making these variants computationally more expensive than both SS-QCQP and the standard MBA method.
\end{remark}

\section{Scalability to Many Constraints}
The main computational bottleneck of the SS-QCQP algorithm lies in solving the QCQP subproblem~\eqref{eq:direction_qcqp} at each iteration.  
When the original problem~\eqref{eq:optimization_problem} contains a large number of inequality constraints ($m \gg 1$), solving this QCQP with all constraints active can be prohibitively expensive.  
However, in many practical applications, only a small subset of constraints is active or nearly active at optimality~\cite{misra2022learning}.  
Thus, enforcing all $m$ constraints in every subproblem is overly conservative.

To improve scalability, we introduce a \emph{specialized active-set strategy} that enforces forward invariance only with respect to constraints that are nearly active.  This reduces the size of the QCQP without compromising theoretical guarantees.  
In particular, we prove that the active-set variant preserves the $O(1/k)$ convergence rate established in~\Cref{thm:convergence_sqcqp} while significantly lowering the per-iteration computational cost.

\subsection{Active-Set Strategy}
We define the $\delta$-\emph{active set} at a point $x$ as
\begin{align}\label{eq:delta_active_set}
    \mathcal{A}_\delta(x) = \{\, i \in [m] ~|~ g_i(x) \ge -\delta \,\}
\end{align}
where $\delta > 0$ is a prescribed threshold. We also define $\mathcal{A}(x)$ to be any set such that $\mathcal{A}_\delta(x) \subseteq \mathcal{A}(x) \subseteq [m]$. The particular choice of $\mathcal{A}(x)$ will be discussed in \Cref{sec:as_choice}.  

At each iteration, instead of considering all $m$ constraints, we propose to compute the search direction by solving a reduced QCQP involving only the nearly active constraints:
\begin{align}\label{eq:direction_qcqp_as}
    \hat{u}(x) := & \arg\min_{u} \frac{1}{2}\|u + \nabla f(x)\|_2^2 \\
    & \text{s.t.} \nabla g_i(x)^\top  u \le -\alpha g_i(x) - w_i\|u\|_2^2,
    \ i \in \mathcal{A}(x). \notag
\end{align}

Next, we will show that the direction $\hat{u}(x)$ inherits the key properties established in \Cref{thm:direction_qcqp}.  Hence, the active-set SS-QCQP maintains the same theoretical guarantees as the full SS-QCQP while achieving substantially improved scalability.

\begin{theorem}\label{thm:direction_qcqp_as}
Let \Cref{ass:assumption} hold, and let $x$ be a feasible point of~\eqref{eq:optimization_problem}.  
Then,
\begin{enumerate}
    \item If $\hat{u}(x) = 0$, then $x$ is a stationary point of~\eqref{eq:optimization_problem}.
    \item The direction $\hat{u}(x)$ satisfies the descent condition
    \begin{align}\label{eq:descent_qcqp_as}
        \nabla f(x)^\top  \hat{u}(x) \le -\|\hat{u}(x)\|_2^2.
    \end{align}
\end{enumerate}
\end{theorem}

\begin{proof}
Since problem~\eqref{eq:direction_qcqp} is strictly feasible, it follows that the reduced subproblem~\eqref{eq:direction_qcqp_as} is also strictly feasible.  
Hence, by the KKT optimality conditions~\cite{boyd2004convex}, there exists a multiplier vector $\lambda(x) \ge 0$ satisfying
\begin{subequations}\label{eq:QCQP_AS_KKT}
\begin{align}
    &\hat{u}(x) + \nabla f(x)
    + \sum_{i \in \mathcal{A}(x)} \lambda_i(x) \big(\nabla g_i(x) + 2 w_i \hat{u}(x)\big) = 0, \label{eq:QCQP_AS_KKT_stationarity} \\
    &\nabla g_i(x)^\top  \hat{u}(x) \le -\alpha g_i(x) - w_i\|\hat{u}(x)\|_2^2, 
    \quad \forall i \in \mathcal{A}(x), \label{eq:QCQP_AS_KKT_primal} \\
    &\lambda_i(x) \ge 0, 
    \quad \forall i \in \mathcal{A}(x), \label{eq:QCQP_AS_KKT_dual} \\
    &\lambda_i(x)\big(\nabla g_i(x)^\top  \hat{u}(x) + \alpha g_i(x) + w_i\|\hat{u}(x)\|_2^2\big) = 0,
    \ \forall i \in \mathcal{A}(x). \label{eq:QCQP_AS_KKT_comp}
\end{align}
\end{subequations}

Comparing the KKT conditions in~\eqref{eq:QCQP_AS_KKT} with those of~\eqref{eq:KKT_conditions}, we observe that when $\hat u(x) = 0$, the pair $(x, \lambda)$ satisfies the original KKT system~\eqref{eq:KKT_conditions}, where
\begin{align}\label{eq:lambda_extension}
    \lambda_i =
    \begin{cases}
        \lambda_i(x), & \text{if } i \in \mathcal{A}(x), \\[2pt]
        0, & \text{otherwise}.
    \end{cases}
\end{align}
Therefore, $x$ is a stationary point of~\eqref{eq:optimization_problem}.
The proof of the descent property~\eqref{eq:descent_qcqp_as} follows identically to the argument presented in \Cref{thm:direction_qcqp}.
\end{proof}

Having defined $\hat{u}(x)$ in \eqref{eq:direction_qcqp_as}, we now consider the corresponding iterative scheme
\begin{align}\label{eq:update_rule_as}
    x^{(k+1)}
    = x^{(k)} + t^{(k)} \hat{u}(x^{(k)}),
\end{align}
where the step size $t^{(k)} > 0$ is selected using the same backtracking conditions as in~\eqref{eq:backtracking_line_search}, ensuring that each new iterate $x^{(k+1)}$ remains feasible and achieves sufficient descent. To guarantee the existence of a uniform lower bound on the accepted step sizes, we impose the following mild assumption.
\begin{assumption}\label{ass:assumption_as}
    The feasible set $\mathcal{F}$ is compact.
\end{assumption}
Under this assumption, the reduced QCQP directions remain uniformly bounded as we show next.
\begin{lemma} \label{lem:qcqp_as_bounded}
    Suppose \eqref{eq:optimization_problem} satisfies \Cref{ass:assumption} and \Cref{ass:assumption_as}, and $w_i \geq \underline{w}$ for some $\underline{w} > 0$ and all $i \in [m]$. Then, the set $\{\hat{u}(x) : x \in \mathcal{F}\}$ is bounded.
\end{lemma}   
\begin{proof}
    Fix any $x \in \mathcal{F}$ and define, for each $i \in [m]$,
    \[
    \Omega_i(x) \triangleq 
    \bigl\{\, u \in \mathbb{R}^n :
    w_i \|u\|_2^2 + \nabla g_i(x)^\top  u + \alpha g_i(x) \le 0
    \,\bigr\}.
    \]
    Dividing both sides by $w_i$ and completing the square, $\Omega_i(x)$ can be represented as
    \[
        \bigl\{
            u \in \mathbb{R}^n 
            : 
            \|u + \tfrac{1}{2w_i}\nabla g_i(x)\|_2^2 
            \leq 
            \tfrac{-\alpha}{w_i} g_i(x) + \|\tfrac{1}{2w_i} \nabla g_i(x)\|_2^2
        \bigl\}.
    \]
    Since $g_i$ is continuously differentiable and $\mathcal{F}$ is compact, both $g_i(x)$ and $\nabla g_i(x)$ are bounded on $\mathcal{F}$.  Hence, there exist finite constants
    \[
    M_i \triangleq \max_{x \in \mathcal{F}} |g_i(x)|, 
    \qquad 
    G_i \triangleq \max_{x \in \mathcal{F}} \|\nabla g_i(x)\|_2.
    \] 
    Using $w_i \ge \underline{w}$ and upper bounds on $g_i(x)$ and $\nabla g_i(x)$, we can write
    \begin{align*}
        \|u + \tfrac{1}{2w_i}\nabla g_i(x)\|_2^2 
        \le R_i^2
    \end{align*}
   where
    $ 
        R_i^2 \triangleq \tfrac{\alpha M_i}{\underline{w}} 
        + \Big(\tfrac{G_i}{2\underline{w}}\Big)^2
    $
    is independent of $x$. Therefore, for all $x \in \mathcal{F}$,
    \[
    \Omega_i(x) \subseteq 
    \bar{\mathcal{B}}\!\left(-\tfrac{1}{2\underline{w}}\nabla g_i(x),\, R_i\right),
    \]
    where $\bar{\mathcal{B}}(c,r)$ denotes the closed Euclidean ball centered at $c$ with radius $r$.
    Because $\mathcal{F}$ is compact and $\nabla g_i$ is continuous, the union
    $\Omega_i \triangleq \bigcup_{x \in \mathcal{F}} \Omega_i(x)$ 
    is bounded. 

    \revise{We distinguish two cases. If $\mathcal{A}(x) = \emptyset$, then $\hat{u}(x)=-\nabla f(x)$ is bounded due to compactness of $\mathcal{F}$. For nonempty $\mathcal{A}(x)$,} since $\hat{u}(x)$ satisfies all active constraints, 
    \[
        \hat{u}(x) \in \bigcap_{i \in \mathcal{A}(x)} \Omega_i(x) 
        \subseteq {\bigcap_{i \in \mathcal{A}(x)}}\Omega_i,
    \]
    implying that $\{\hat{u}(x) : x \in \mathcal{F}\}$ is bounded.
    %
\end{proof}

\begin{lemma}\label{lem:qcqp_as_step_size}
Suppose~\eqref{eq:optimization_problem} satisfies \Cref{ass:assumption}, 
and let $x$ be a feasible point of~\eqref{eq:optimization_problem}. 
Assume $\gamma \in (0,1)$ and $w_i \ge \underline{w}$ for all $i \in [m]$, for some $\underline{w} > 0$. Then there exists a constant $M > 0$ such that the backtracking conditions in~\Cref{eq:backtracking_line_search} hold for all $t \in [0,\underline{t}]$, where
\begin{align}\label{eq:active_step_lower_bound}
    \underline{t}
    = \min\Bigg\{
        \tfrac{2(1-\gamma)}{L_f},\,
        1,\,
        \tfrac{\delta}{M},\,
        \tfrac{1}{\alpha},\,
        \tfrac{2\underline{w}}{\max_{i \in [m]} L_i}
    \Bigg\}.
\end{align}
\end{lemma}

\begin{proof}
    From \Cref{thm:direction_qcqp_as}, we have $\nabla f(x)^\top  \hat{u}(x) \le -\|\hat{u}(x)\|_2^2$. 
    Hence, for all $t \in [0, \tfrac{2(1-\gamma)}{L_f}]$, the Armijo condition is satisfied, i.e.,
    \begin{align}
        & f(x + t\hat{u}(x)) 
        \le f(x) + t\nabla f(x)^\top  \hat{u}(x)
        + \tfrac{L_f t^2}{2}\|\hat{u}(x)\|_2^2 \\
        & \le f(x) + \gamma t\nabla f(x)^\top  \hat{u}(x). \notag
    \end{align}
    The proof follows the same steps as in \Cref{lem:qcqp_step_size} for the full QCQP case.

    For feasibility, consider two cases. 
    
    \textit{Case 1:} $i \in \mathcal{A}(x)$.   Then, by the constraint in \eqref{eq:direction_qcqp_as},  $g_i(x + t\hat{u}(x)) \le 0$ for all  $t \in [0, \min\{\tfrac{1}{\alpha}, \tfrac{2\underline{w}}{L_i}\}]$, as in \Cref{lem:qcqp_step_size}. 

    \textit{Case 2:} $i \notin \mathcal{A}(x)$.  
    Since $\mathcal{A}_\delta(x) \subseteq \mathcal{A}(x)$, we have $g_i(x) \le -\delta$. Define
    \[
        s_i(t) \triangleq g_i(x) + t\nabla g_i(x)^\top  \hat{u}(x)
        + \tfrac{L_i t^2}{2}\|\hat{u}(x)\|_2^2,
    \]
    which is a quadratic upper bound on $g_i(x + t\hat{u}(x))$. The largest nonnegative root $\eta_i$ of $s_i(t) = 0$ is given by
    \begin{align*}
    \eta_i 
    = \tfrac{
        -\nabla g_i(x)^\top  \hat{u}(x)
        + \sqrt{
            (\nabla g_i(x)^\top  \hat{u}(x))^2
            - 2L_i g_i(x)\|\hat{u}(x)\|_2^2
        }
    }{L_i \|\hat{u}(x)\|_2^2}.
    \end{align*}
    By construction, $s_i(0) = g_i(x) \le -\delta < 0$, $s_i(\eta_i) = 0$, and for all $t \in [0,\eta_i]$, 
    \[
        g_i(x + t\hat{u}(x)) \le s_i(t) \le 0.
    \]
    Applying the mean value theorem, there exists $\theta_i \in [0, \eta_i]$ such that
    \begin{align*}
        -g_i(x)
        &= s_i(\eta_i) - s_i(0)
         = s_i'(\theta_i)\eta_i \\
        &= \eta_i \big[\nabla g_i(x)^\top  \hat{u}(x)
            + L_i\theta_i\|\hat{u}(x)\|_2^2\big] \\ &
         \le \eta_i \big[|\nabla g_i(x)^\top  \hat{u}(x)|
            + L_i \theta_i \|\hat{u}(x)\|_2^2\big].
    \end{align*}
    If $\eta_i \le 1$, then since $\theta_i \leq \eta_i$, we can write 
    \[
        -g_i(x) \le \eta_i \big[|\nabla g_i(x)^\top  \hat{u}(x)|
            + L_i\|\hat{u}(x)\|_2^2\big].
    \]
    Since $-g_i(x) \ge \delta$, we have
    \[
        \eta_i \ge 
        \frac{\delta}{
            |\nabla g_i(x)^\top  \hat{u}(x)|
            + L_i\|\hat{u}(x)\|_2^2
        }.
    \] 
    \revise{Since $\{\hat{u}(x) : x \in \mathcal{F}\}$ is bounded following \Cref{lem:qcqp_as_bounded}, and $\{\nabla g_i(x): x \in \mathcal{F}\}$ is bounded by Lipschitz smoothness of $\nabla g_i(\cdot)$ and compactness of $\mathcal{F}$},
    there exists a constant $M > 0$ such that
    \[
        |\nabla g_i(x)^\top  \hat{u}(x)|
        + L_i\|\hat{u}(x)\|_2^2 \le M,
        \qquad \forall\, x \in \mathcal{F},~ i \in [m].
    \]
    Hence, $\eta_i \ge \tfrac{\delta}{M}$.  Eliminating the temporary assumption $\eta_i \le 1$, we obtain
    \[
        \eta_i \ge \min\Big\{1,\, \tfrac{\delta}{M}\Big\}.
    \]
    Combining all cases, we conclude that:
    \begin{itemize}
        \item The Armijo condition holds for all 
        $t \in [0, \tfrac{2(1-\gamma)}{L_f}]$;
        \item For all $i \in \mathcal{A}(x)$, 
        $g_i(x + t\hat{u}(x)) \le 0$ for 
        $t \in [0, \min\{\tfrac{1}{\alpha}, \tfrac{2\underline{w}}{L_i}\}]$;
        \item For all $i \notin \mathcal{A}(x)$, 
        $g_i(x + t\hat{u}(x)) \le 0$ for 
        $t \in [0, \min\{1, \tfrac{\delta}{M}\}]$.
    \end{itemize}
    Therefore, the backtracking conditions~\eqref{eq:backtracking_line_search} hold for all 
    $t \in [0, \underline{t}]$, where $\underline{t}$ is given by~\eqref{eq:active_step_lower_bound}.
\end{proof}
\Cref{thm:direction_qcqp_as} and \Cref{lem:qcqp_as_step_size} prove similar results to \Cref{thm:direction_qcqp} and \Cref{lem:qcqp_step_size}, respectively, enabling us to extend \Cref{alg:SQCQP} with active set strategy while preserving theoretical convergence guarantees. The algorithm is summarized in \Cref{alg:SQCQP-AS}.

\begin{algorithm}[H]
\caption{SS-QCQP with Active Set Strategy (SS-QCQP-AS) for solving~\eqref{eq:optimization_problem}}
\label{alg:SQCQP-AS}
\begin{algorithmic}[1]
    \State \textbf{Input:} initial feasible point $x^{(0)}$ such that $g_i(x^{(0)}) \le 0$ for all $i\in [m]$; 
    parameters {\color{darkblue}$\delta>0$}, $\gamma \in (0,1)$, \revise{$\beta \in (0,1)$}, $\underline{w} > 0$, $\alpha > 0$, and tolerance $\epsilon > 0$.
    \For{$k = 0,1,2,\ldots$}
        \State Choose adaptive parameters $\{\,w_i^{(k)} \ge \underline{w} : i\in [m]\}$.
        \State {\color{darkblue}Choose active set such that $\mathcal{A}_\delta(x) \subseteq \mathcal{A}(x)$\Comment{Active set strategy}}
        \State Compute the search direction $u = \hat{u}(x^{(k)})$ by solving~{\color{darkblue}\eqref{eq:direction_qcqp_as}} with $w_i = w_i^{(k)}$.
        \If{$\|u\|_2 \le \epsilon$}
            \State \Return $x^{(k)}$ \Comment{Convergence achieved}
        \EndIf
        \State Initialize step size $t^{(k)} = 1$.
        \While{
            $\exists i \in [m]$ such that 
            $g_i(x^{(k)} + t^{(k)} u) > 0$
            \textbf{or}
            $f(x^{(k)} + t^{(k)} u) > f(x^{(k)}) + \gamma t^{(k)} \nabla f(x^{(k)})^\top  u$
        }
            \State \revise{$t^{(k)} := \beta t^{(k)}$ }\Comment{Backtracking line search}
        \EndWhile
        \State Update $x^{(k+1)} := x^{(k)} + t^{(k)} u$.
    \EndFor
\end{algorithmic}
\end{algorithm}

\begin{theorem}\label{thm:convergence_sqcqp_as}
Let \Cref{ass:assumption} and \Cref{ass:assumption_as} hold for ~\eqref{eq:optimization_problem}. 
Then, there exists a constant $M > 0$ such that, for all iterations $k = 0,1,2,\ldots$ generated by \revise{\Cref{alg:SQCQP-AS}}, the following bound holds:
\revise{
\begin{align}\label{eq:sqcqp_as_convergence_bound}
    \min_{i = 0,\ldots,k} \|\hat{u}(x^{(i)})\|_2^2
    \le 
    \frac{f(x^{(0)}) - f^\star}{\gamma\,\min\{1, \beta \underline{t}\}(k+1)},
\end{align}
}
where
\begin{align}\label{eq:sqcqp_as_step_bound}
    \underline{t}
    = \min\bigg\{
        \tfrac{2(1-\gamma)}{L_f},\,
        1,\,
        \tfrac{\delta}{M},\,
        \tfrac{1}{\alpha},\,
        \tfrac{2\underline{w}}{\max_{i \in [m]} L_i}
    \bigg\}.
\end{align}
\end{theorem}
\begin{proof}
    Given a lower bound on step size from \Cref{lem:qcqp_as_step_size}, the proof is the same as in \Cref{thm:convergence_sqcqp}.
\end{proof}




\section{Practical Implementation}

\subsection{Efficient Computation of the Search Direction}
The optimization problem~\eqref{eq:direction_qcqp} (or its reduced form~\eqref{eq:direction_qcqp_as}) is convex and can be solved efficiently using modern conic solvers such as ECOS~\cite{domahidi2013ecos} and Clarabel~\cite{goulart2024clarabel}, after reformulating the quadratic constraints as second-order cone (SOC) constraints.  
While these solvers are robust and well optimized, the efficiency of each iteration in SS-QCQP depends critically on how the problem is expressed in conic form. To improve practical performance, we exploit a structural feature of~\eqref{eq:direction_qcqp}, where both the objective and the constraints are quadratic functions with identity Hessians.  For simplicity, we describe the formulation when all $m$ constraints are included in the subproblem; the same approach applies directly to the reduced QCQP~\eqref{eq:direction_qcqp_as}.

Introducing an auxiliary variable $s$, problem~\eqref{eq:direction_qcqp} can be equivalently written as
\begin{align}\label{eq:direction_qcqp_transformed}
    \hat{u}(x), \hat{s}(x)
    &= \arg\min_{u,\,s}~ \tfrac{1}{2}\|u + \nabla f(x)\|_2^2 \\
    \text{s.t.}~&
    \nabla g_i(x)^\top u + w_i s + \alpha g_i(x) \le 0, 
    \quad \forall i \in [m], \notag\\
    &\|u\|_2^2 \le s. \notag
\end{align}
See ~\Cref{lem:qcqp_epigraph} below for proof of equivalence. Although the transformation~\eqref{eq:direction_qcqp_transformed} is conceptually straightforward, it significantly simplifies the problem representation for off-the-shelf conic solvers. 
For example, Clarabel~\cite{goulart2024clarabel} solves generic conic programs of the form
\begin{align}
    \min_{x,\,s}~& \tfrac{1}{2} x^\top P x + q^\top x \\
    \text{s.t.}~& A x + s = b, \notag\\
                & s \in \mathcal{K}, \notag
\end{align}
where $\mathcal{K}$ denotes a convex cone (e.g., a second-order cone).  
When modeling~\eqref{eq:direction_qcqp} through high-level interfaces such as CVXPY~\cite{diamond2016cvxpy}, the auxiliary-variable form~\eqref{eq:direction_qcqp_transformed} leads to a sparser coefficient matrix $A$ and a more compact representation, yielding consistent reductions in solver runtime across iterations.

\begin{lemma}\label{lem:qcqp_epigraph}
Problems~\eqref{eq:direction_qcqp} and~\eqref{eq:direction_qcqp_transformed} are equivalent in the sense that, if 
$(\hat{u}(x), \hat{s}(x))$ is an optimal solution of~\eqref{eq:direction_qcqp_transformed}, then 
$\hat{u}(x)$ is the unique optimal solution of~\eqref{eq:direction_qcqp}.
\end{lemma}

\begin{proof}
Let $p_1^\star$ and $p_2^\star$ denote the optimal values of~\eqref{eq:direction_qcqp} 
and~\eqref{eq:direction_qcqp_transformed}, respectively.  For any $(u,s)$ feasible to~\eqref{eq:direction_qcqp_transformed}, 
the first constraint implies that $u$ is feasible for~\eqref{eq:direction_qcqp}.  
Since both problems share the same objective, it follows that 
$p_1^\star \le p_2^\star$. Conversely, for any $u$ feasible to~\eqref{eq:direction_qcqp}, 
the pair $(u, \|u\|_2^2)$ is feasible for~\eqref{eq:direction_qcqp_transformed} 
and yields the same objective value, implying $p_2^\star \le p_1^\star$.  
Thus, $p_1^\star = p_2^\star$.  

Consequently, if $(\hat{u}(x), \hat{s}(x))$ is an optimal solution of~\eqref{eq:direction_qcqp_transformed}, 
then $\hat{u}(x)$ attains the same optimal value in~\eqref{eq:direction_qcqp}, 
and hence is optimal for that problem.  
Since the objective of~\eqref{eq:direction_qcqp} is strongly convex, 
the solution $\hat{u}(x)$ is unique.
\end{proof}

\subsection{Hyperparameter Choice}\label{sec:w_choice}
Another important practical aspect of the proposed SS-QCQP algorithm is the selection of the parameters $\alpha$ and $\{w_i\}_{i=1}^m$ in~\eqref{eq:direction_qcqp}.  
As illustrated in~\Cref{fig:directions}, when $w_i$ is small, the direction $\hat{u}(x)$ becomes nearly tangent to the constraint boundary, leading to a small step size required to maintain feasibility.  
Conversely, when $w_i$ is large, $\hat{u}(x)$ deviates significantly from the negative gradient direction, resulting in a smaller decrease in the objective value.  
Therefore, the choice of $w_i$ critically affects the empirical performance of~\Cref{alg:SQCQP}.

Inspired by the MBA method~\cite{auslender2010moving}, we fix $\alpha = 1$ and set $w_i \approx \frac{L_i}{2}$.  
However, since $L_i$ is typically unknown in practice, we adaptively estimate it during iterations.  
We initialize each $w_i^{(0)} = \underline{w}$ with a small positive value (e.g., $10^{-3}$), and update it according to
\begin{align}\label{eq:adaptive_w_update}
    w_i^{(k+1)} 
    = \max\!\left\{
        w_i^{(k)},
        \tfrac{
            \|\nabla g_i(x^{(k+1)}) - \nabla g_i(x^{(k)})\|_2
        }{
           2 \|x^{(k+1)} - x^{(k)}\|_2
        }
    \right\},
    \quad \forall i \in [m].
\end{align}
Since the sequence $\{w_i^{(k)}\}$ is nondecreasing, the convergence guarantees established in \Cref{thm:convergence_sqcqp} and \Cref{thm:convergence_sqcqp_as} remain valid.

\subsection{Active Set Choice}\label{sec:as_choice}
In~\Cref{lem:qcqp_as_step_size} and~\Cref{thm:convergence_sqcqp_as}, we only required that 
$\mathcal{A}_\delta(x) \subseteq \mathcal{A}(x)$, allowing flexibility in the choice of $\mathcal{A}$.  A natural and straightforward option is to set $\mathcal{A}(x) = \mathcal{A}_\delta(x)$.  
However, with this choice, even small perturbations in $x$ can cause significant changes in $\mathcal{A}_\delta(x)$, which in turn may induce large variations in the direction $\hat{u}(x)$.  
Such discontinuities can produce zigzagging trajectories and degrade the empirical performance of the algorithm. To improve stability, we propose to use 
    $\mathcal{A}(x) = \mathcal{A}_\delta(x) \cup \mathcal{T}_q(x)$, where
\begin{align}\label{eq:quantile_active_set}
    \mathcal{T}_q(x) = \{i : g_i(x) \text{ is among top $q \%$ of } \{g_i(x) : i\in [m]\}\}.
\end{align}
This construction smooths transitions in $\mathcal{A}(x)$ by including a small fraction of nearly active constraints, effectively balancing scalability and stability.

\section{Numerical Experiments}




\subsection{Multi-Car Navigation}
In this section, we demonstrate the effectiveness of the proposed SS-QCQP algorithm on an optimal control problem involving multi-agent navigation.  
The objective is to steer four vehicles from their initial positions to their designated goal locations while avoiding collisions with both static obstacles and other vehicles. The problem setup is similar to that in~\cite{wu2025accelerated}.

The problem is summarized as follows,
\begin{align} \label{eq:MPC_experiment}
    \min_{X, U}. ~& \sum_{t=0}^{N-1} l(X(t), U(t)) + V(X(N)) \\
    \text{s.t.} ~& U_i(t) \in \mathcal{U} ~\forall t \in [N-1], i \in [4],  \notag\\
                ~& X_i(t) \in \mathcal{X} ~\forall t \in [N], i \in [4], \notag  \\
                ~& X_i(t+1) = f(X_i(t), U_i(t)) ~\forall t \in [N-1], i \in [4], \notag \\
                ~& X_i(0) = X_i^s ~\forall i \in [4]\notag \\
                ~~& g_j(X_i(t)) \leq 0 ~ \forall i \in [4], j \in [3], t \in [N], \notag \\
                ~~& \|X_i(t)[:2] - X_j(t)[:2\|_2^2 \geq 1 ~ \forall i \in [4], j \in [4], t \in [N] \notag
\end{align}
where $X_i(t) = [x_i(t),\, y_i(t),\, \theta_i(t)]^\top$ is the state of the $i$-th car at time step $t$, consisting of the $x$-position, $y$-position, and orientation, and $U_i(t) = [v_i(t),\, w_i(t)]^\top$ is its control input, consisting of the linear and angular velocities.  
The discrete-time dynamics, shared by all cars, are given by
\[
    f(X_i(t), U_i(t)) =
    \begin{bmatrix}
        x_i(t) + v_i(t)T\cos\theta_i(t) - \Delta x \\
        y_i(t) + v_i(t)T\sin\theta_i(t) \\
        \theta_i(t) + w_i(t)T
    \end{bmatrix}.
\]
where $T=\Delta x = 0.03$.
Here, $X(t) = [X_i(t)^\top]_{i\in[4]}^\top$ stacks the states of all four cars at time~$t$, and 
$X = [X(t)^\top]_{t\in[N]}$ and $U = [U(t)^\top]_{t\in[N-1]}$ denote the stacked states and control inputs over the entire horizon.

 The initial and target states are specified as follows. The starting state of the $i$-th \revise{vehicle} is denoted by $X_i^s$, with
 \[
\begin{aligned}
X_1^s &= [-2,\,-2,\,0]^\top, &\quad X_2^s &= [-3,\,-1,\,0]^\top,\\
X_3^s &= [-3,\,-3,\,0]^\top, &\quad X_4^s &= [-1,\,-3,\,0]^\top, 
\end{aligned}
\]
and the corresponding target (desired) states are
\[
\begin{aligned}
X_1^d &= [2,\,3,\,0]^\top, &\quad X_2^d &= [3,\,2,\,0]^\top,\\
X_3^d &= [2,\,1,\,0]^\top, &\quad X_4^d &= [1,\,2,\,0]^\top.
\end{aligned}
\]
The desired steady-state input is $U_i^d = [1,\,0]^\top$ for all $i \in [4]$. The stage cost is defined as
\[
    l(X(t), U(t)) 
    = \sum_{i \in [4]} 
        \big(
            \|X_i(t) - X_i^d\|_2^2 
            + 0.01\,\|U_i(t) - U_i^d\|_2^2
        \big),
\]
and the terminal cost is
\[
    V(X(N)) 
    = \sum_{i \in [4]} 
        \|X_i(N) - X_i^d\|_{P_i}^2,
\]
where each $P_i$ is obtained from the discrete-time algebraic Riccati Equation 
\begin{small}
\begin{align*}
    A_i^\top P_i A_i - P_i - (A_i^\top P_i B_i)(R + B_i^\top P_i B_i)^{-1}(B_i^\top P_i A_i) + Q = 0
\end{align*}
\end{small}
with $A_i = \nabla_1 f(X_i^d, U_i^d)$, $B_i = \nabla_2 f(X_i^d, U_i^d)$, $Q=I$ and $R=0.01 I$. The state and input constraints are defined as
\[
\begin{aligned}
\mathcal{X} &= 
\{(x, y, \theta) : |x| \le 3.7, |y| \le 3.7,~ |\theta| \le \pi\},\\
\mathcal{U} &= 
\{(v, w) : -5 \le v \le 12,~  |w| \le \tfrac{3}{2}\pi\}.
\end{aligned}
\]
Obstacle avoidance is enforced through the constraint
\[
    g_j(X_i(t)) = r_j^2 - \|X_i(t)\revise{[:2]} - c_j\|_2^2 \le 0,
\]
which models circle obstacles centered at $c_1 = [-1,-1]^\top, c_2 = [1,0]^\top, c_3 = [0,1]^\top$ with radii $r_1 = 1, r_2 = 0.5, r_3 = 0.5$. 

By unrolling the dynamics from the starting state and applying the control inputs, \eqref{eq:MPC_experiment} can be transformed into an optimization problem in form of \eqref{eq:optimization_problem}. We set the planning horizon to $N=40$, resulting in a total number of $320$ variables and $2320$ constraints.


\revise{SS-QCQP and SS-QCQP-AS}
were implemented in CasADi~\cite{andersson2019casadi}, and the corresponding subproblems~\eqref{eq:direction_qcqp} and~\eqref{eq:direction_qcqp_as} were solved using Clarabel~\cite{goulart2024clarabel} through the CVXPY~\cite{diamond2016cvxpy} interface.  
In both methods, the parameters were initialized as $w_i = 10^{-3}$ and adaptively updated according to~\eqref{eq:adaptive_w_update}.  
For the active-set strategy, we used $\mathcal{A}(x) = \mathcal{A}_{0.5}(x) \cup \mathcal{T}_5(x)$, where $\mathcal{A}_\delta(x)$ and $\mathcal{T}_q(x)$ are defined in~\eqref{eq:delta_active_set} and~\eqref{eq:quantile_active_set}, respectively. The experiments are conducted on a 2020 MacBook Pro equipped with an Apple M1 CPU and 16 GB of memory.

We compared SS-QCQP and SS-QCQP-AS against IPOPT~\cite{wachter2006implementation} (as called from CasADi) and CasADi's built-in SQP solver.  
\Cref{fig:1_traj} illustrates the position trajectories of the four agents obtained using SS-QCQP, SS-QCQP-AS, IPOPT, and SQP.  
The dotted lines, stars, and triangles represent the trajectories, initial positions, and target positions of the agents, respectively, while the gray circles denote obstacles.  
It can be observed that \Cref{alg:SQCQP}, \Cref{alg:SQCQP-AS}, and SQP yield similar near-optimal trajectories.  
\Cref{fig:1_dist_cars} further depicts the pairwise distances among the four agents over the planning horizon obtained using SS-QCQP.

\begin{figure}[ht]
\centering
\begin{subfigure}[t]{0.24\textwidth}
\includegraphics[scale=0.37,trim=0 0 0 0, clip]{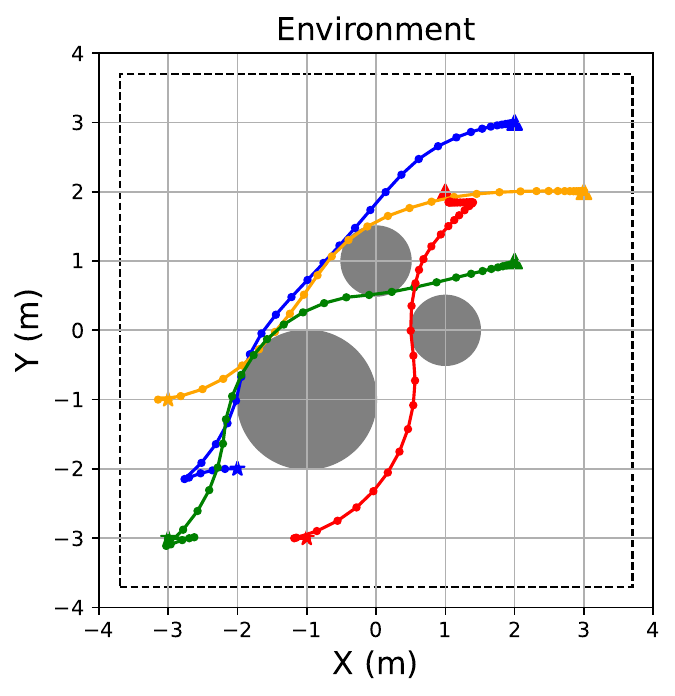}
\caption{SS-QCQP}
\end{subfigure}
\begin{subfigure}[t]{0.24\textwidth}
\includegraphics[scale=0.37,trim=0 0 0 0, clip]{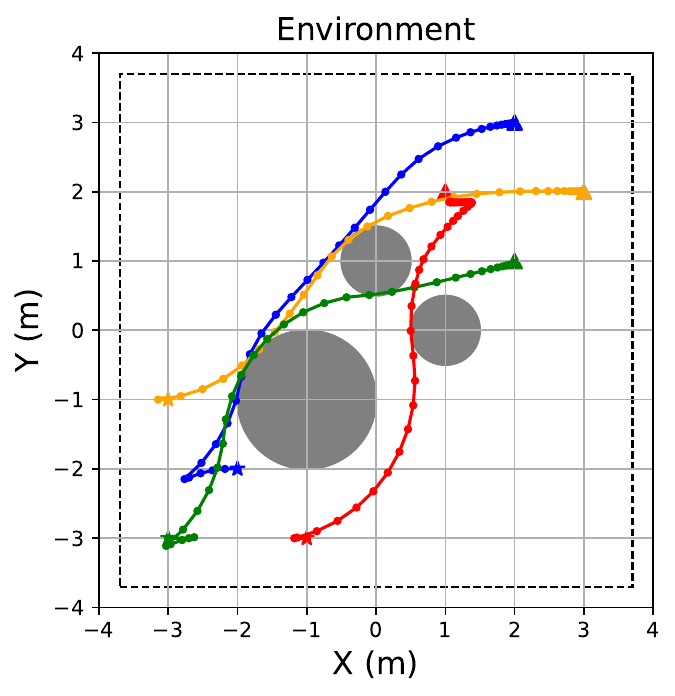}
\caption{SS-QCQP-AS}
\end{subfigure}

\begin{subfigure}[t]{0.24\textwidth}
\includegraphics[scale=0.37,trim=0 0 0 0, clip]{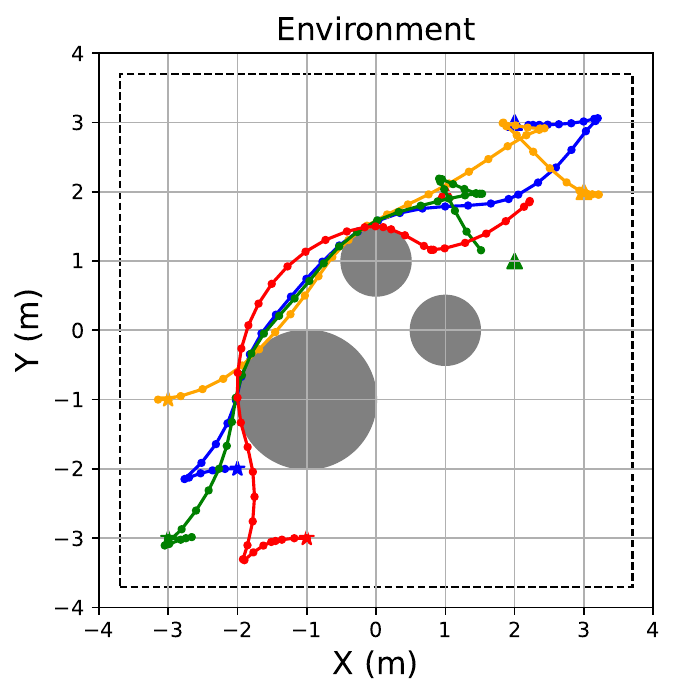}
\caption{IPOPT}
\end{subfigure}
\begin{subfigure}[t]{0.24\textwidth}
\includegraphics[scale=0.37,trim=0 0 0 0, clip]{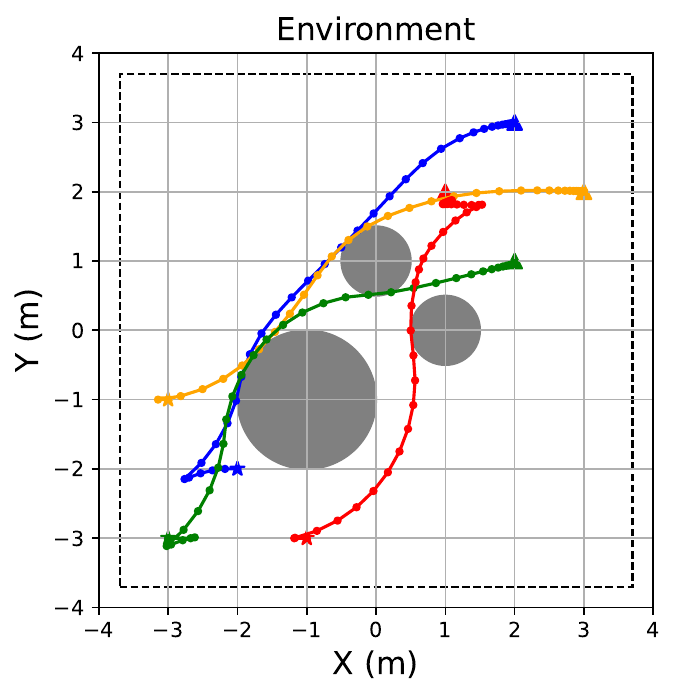}
\caption{SQP}
\end{subfigure}
\caption{Position trajectories of four agents obtained using SS-QCQP, SS-QCQP-AS, IPOPT, and SQP.}
\label{fig:1_traj}
\end{figure}

\begin{figure*}[ht]
  \centering
  \begin{subfigure}[t]{0.2425\textwidth}
    \includegraphics[width=\textwidth, trim=0 0 0 0]{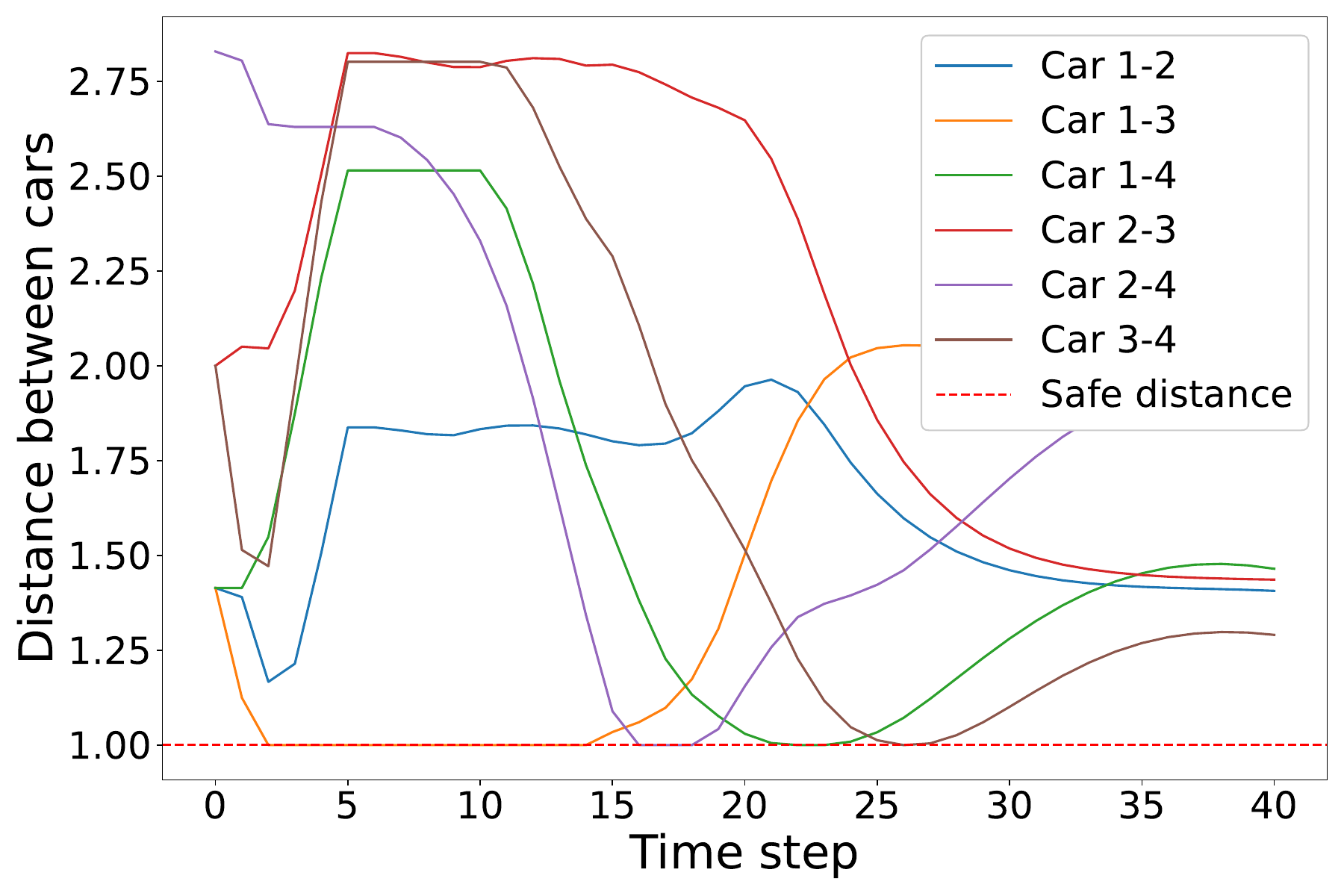}
    \caption{}
    \label{fig:1_dist_cars}
  \end{subfigure}
  \hfill
  \begin{subfigure}[t]{0.2425\textwidth}
    \includegraphics[width=\textwidth, trim=0 0 0 0]{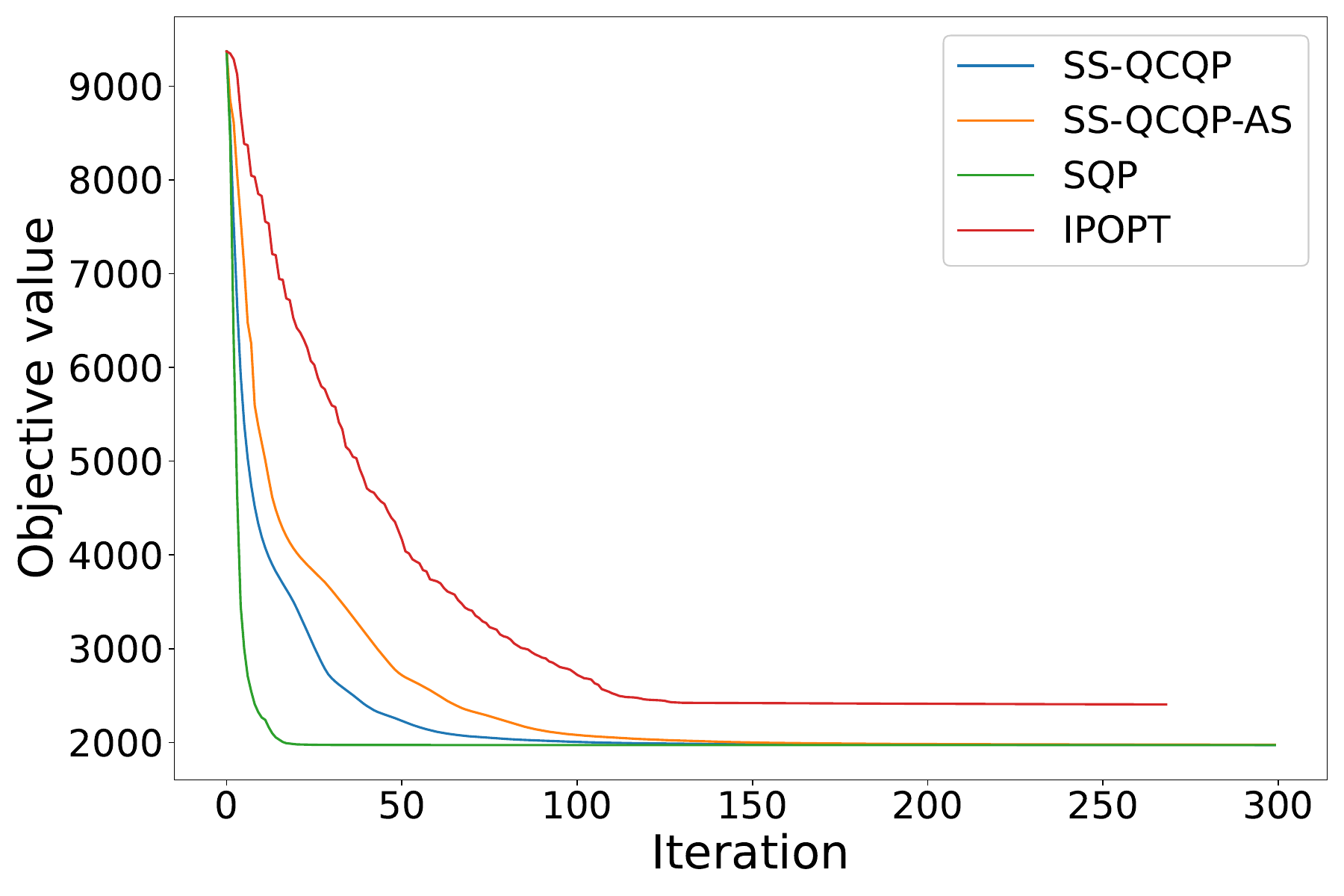}
    \caption{}
    \label{fig:1_fvals}
  \end{subfigure}
  \hfill
  \begin{subfigure}[t]{0.2425\textwidth}
    \includegraphics[width=\textwidth, trim=0 0 0 0]{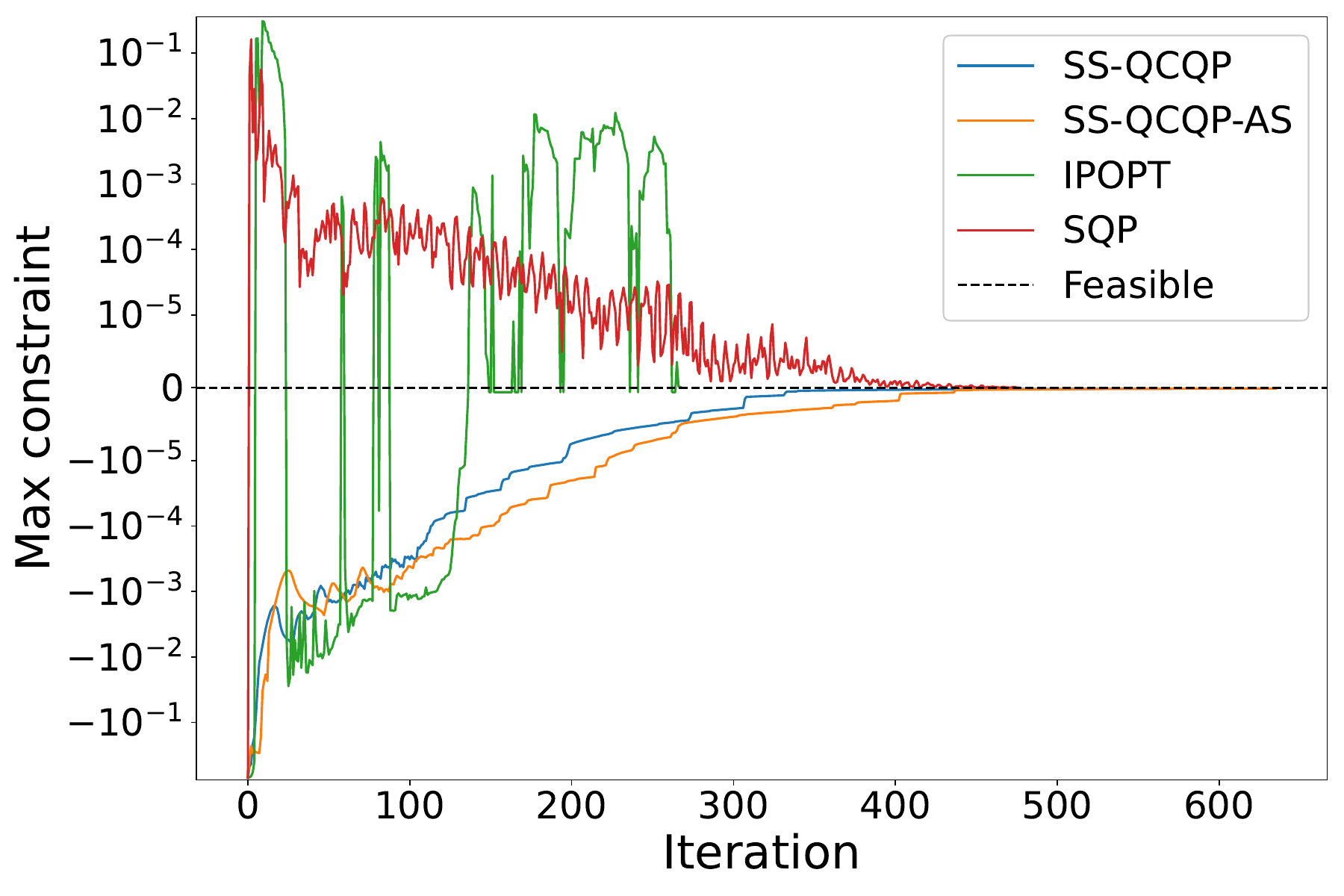}
    \caption{}
    \label{fig:1_feasibility}
  \end{subfigure}
  \hfill
  \begin{subfigure}[t]{0.2425\textwidth}
    \includegraphics[width=\textwidth, trim=0 0 0 0]{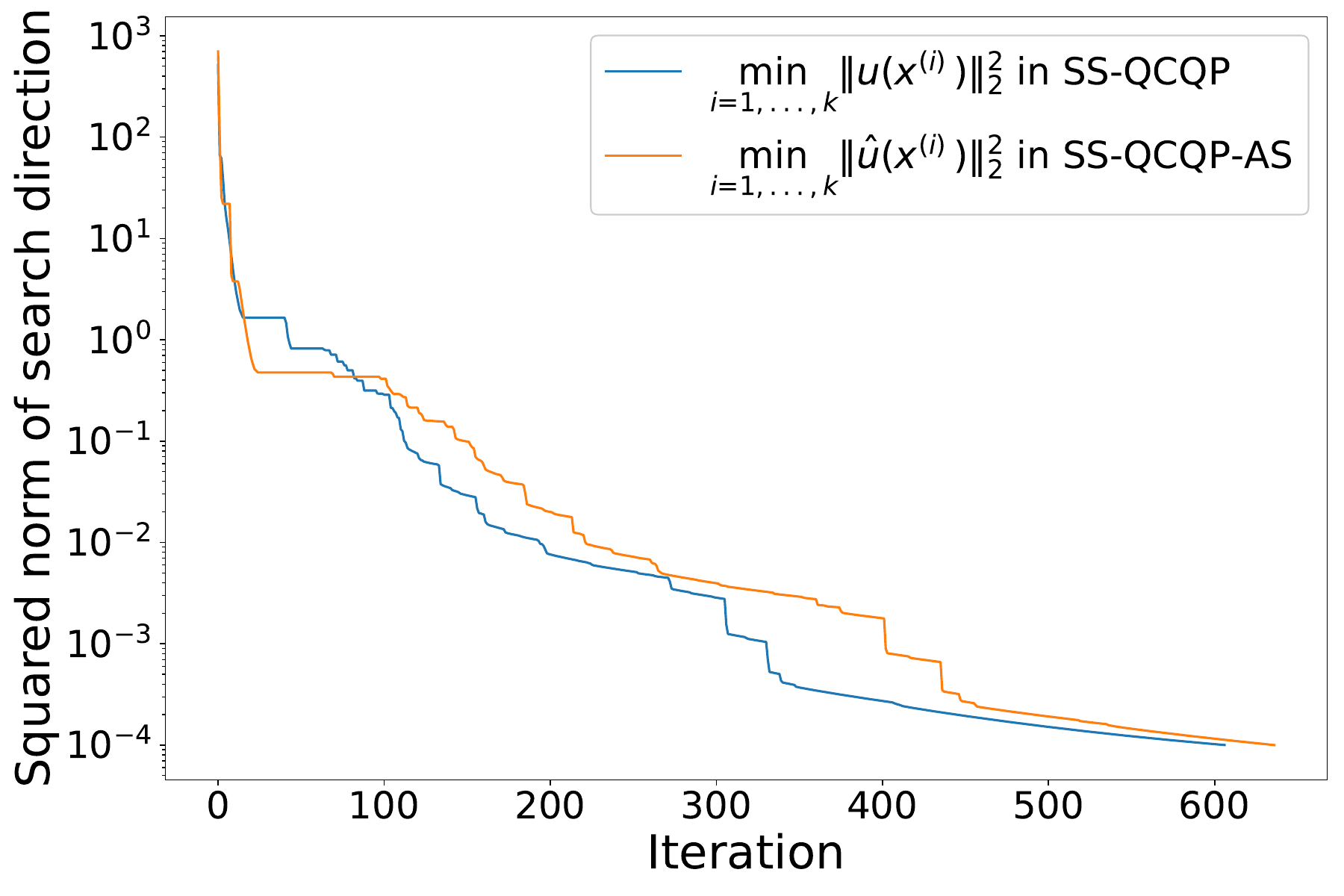}
    \caption{}
    \label{fig:1_u_norms}
  \end{subfigure}
  \caption{(a): Pairwise distances among the four agents over the planning horizon solved by SS-QCQP. (b): Objective of SS-QCQP, SS-QCQP-AS, IPOPT, and SQP along iterations. (c): Maximum constraint of SS-QCQP, SS-QCQP-AS, IPOPT, and SQP along iterations. (d): $\min_{i=1,...,k}\|u(x^{(i)})\|_2^2$ in SS-QCQP and $\min_{i=1,...,k}\|\hat{u}(x^{(i)})\|_2^2$ in SS-QCQP-AS along iterations.}

  \label{fig:four-wide}
\end{figure*}

\Cref{fig:1_fvals} and \Cref{fig:1_feasibility} plot, for each algorithm, the objective value and the maximum constraint violation across iterations, respectively.  
Furthermore, \Cref{fig:1_u_norms} shows the minimum squared norm of the search directions across iterations, which converges at the theoretical rate of $O(1/k)$ established in~\Cref{thm:convergence_sqcqp,thm:convergence_sqcqp_as}.  
Finally, \Cref{fig:1_active_constraints} shows the number of nearly active constraints at each iteration, illustrating that fewer than $10\%$ of all constraints are nearly active and enter the QCQP subproblem. The runtime of \Cref{alg:SQCQP} is 186.53 seconds, whereas \Cref{alg:SQCQP-AS} achieves a comparable solution in 49.47 seconds.

It can be observed that while the SQP method decreases the objective value more rapidly than both SS-QCQP and SS-QCQP-AS, it exhibits small constraint violations before convergence.  
In contrast, both SS-QCQP and SS-QCQP-AS remain feasible throughout all iterations, albeit with a slower decrease in the objective value.  
This property enables early termination with feasibility guarantees.  
It is worth noting that both SQP and IPOPT are second-order methods (in this experiment, CasADi’s built-in SQP employs an L-BFGS Hessian approximation), whereas SS-QCQP and SS-QCQP-AS are first-order algorithms that achieve comparable solution quality while ensuring feasibility at every iteration.




\begin{figure}[ht]
\centering
\includegraphics[scale=0.28, trim=0 0 0 0]{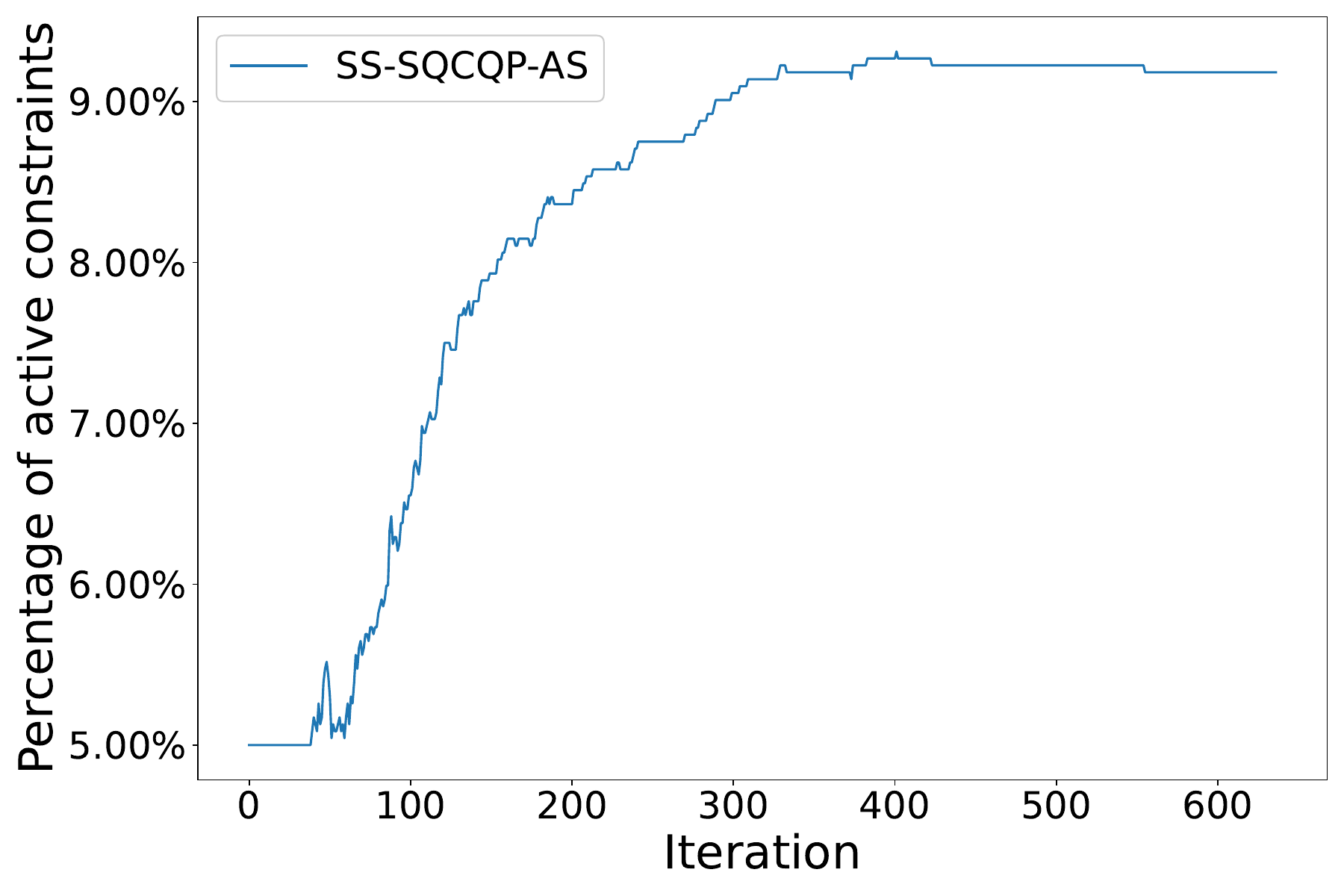}
\caption{\revise{Percentage} of active constraints in \Cref{alg:SQCQP-AS} along the iterations. The total number of constraints is 2320.}
\label{fig:1_active_constraints}
\end{figure}

\section{Conclusion}
This paper introduced the Safe Sequential QCQP (SS-QCQP) algorithm, a first-order framework for solving inequality-constrained nonlinear optimization problems with guaranteed feasibility at every iteration.  
The method is derived from a continuous-time dynamical system whose vector field is obtained by solving a convex quadratically constrained quadratic program (QCQP) that enforces strict descent and forward invariance of the feasible set.  
A safeguarded discretization with adaptive step-size selection yields a discrete-time algorithm that maintains feasibility and achieves a provable $O(1/k)$ convergence rate.  

To improve scalability, we developed an active-set variant (SS-QCQP-AS) that selectively enforces invariance for nearly active constraints, significantly reducing computational cost without compromising convergence guarantees.  Extensive numerical experiments on a multi-agent optimal control problem confirmed that both SS-QCQP and SS-QCQP-AS maintain feasibility, exhibit the predicted convergence behavior, and deliver performance comparable to second-order solvers such as SQP and IPOPT.  

Future work will explore the integration with model predictive control (MPC) schemes, as well as acceleration strategies using higher-order derivatives.

\printbibliography

@article{raghunathan2025constrained,
  title={Constrained Optimization From a Control Perspective via Feedback Linearization},
  author={Raghunathan, Arvind and Shamma, Jeff and Li, Na and others},
  journal={arXiv preprint arXiv:2503.12665},
  year={2025}
}

@article{10.5555/3586589.3586845,
author = {Muehlebach, Michael and Jordan, Michael I.},
title = {On constraints in first-order optimization: a view from non-smooth dynamical systems},
year = {2022},
issue_date = {January 2022},
publisher = {JMLR.org},
volume = {23},
number = {1},
issn = {1532-4435},
journal = {J. Mach. Learn. Res.},
month = jan,
articleno = {256},
numpages = {47}
}

@article{allibhoy2023control,
  title={Control-barrier-function-based design of gradient flows for constrained nonlinear programming},
  author={Allibhoy, Ahmed and Cort{\'e}s, Jorge},
  journal={IEEE Transactions on Automatic Control},
  volume={69},
  number={6},
  pages={3499--3514},
  year={2023},
  publisher={IEEE}
}

@inproceedings{allibhoy2021anytime,
  title={Anytime solution of constrained nonlinear programs via control barrier functions},
  author={Allibhoy, Ahmed and Cort{\'e}s, Jorge},
  booktitle={2021 60th IEEE Conference on Decision and Control (CDC)},
  pages={6527--6532},
  year={2021},
  organization={IEEE}
}

@phdthesis{maratos1978exact,
  title={Exact penalty function algorithms for finite dimensional and control optimization problems},
  author={Maratos, Nicholas},
  year={1978},
  school={Imperial College London (University of London)}
}

@article{auslender2010moving,
  title={A moving balls approximation method for a class of smooth constrained minimization problems},
  author={Auslender, Alfred and Shefi, Ron and Teboulle, Marc},
  journal={SIAM Journal on Optimization},
  volume={20},
  number={6},
  pages={3232--3259},
  year={2010},
  publisher={SIAM}
}

@book{boyd2004convex,
  title={Convex optimization},
  author={Boyd, Stephen P and Vandenberghe, Lieven},
  year={2004},
  publisher={Cambridge university press}
}

@book{blanchini2008set,
  title={Set-theoretic methods in control},
  author={Blanchini, Franco and Miani, Stefano and others},
  volume={78},
  year={2008},
  publisher={Springer}
}

@article{blanchini1999set,
  title={Set invariance in control},
  author={Blanchini, Franco},
  journal={Automatica},
  volume={35},
  number={11},
  pages={1747--1767},
  year={1999},
  publisher={Elsevier}
}

@article{dawson2023safe,
  title={Safe control with learned certificates: A survey of neural lyapunov, barrier, and contraction methods for robotics and control},
  author={Dawson, Charles and Gao, Sicun and Fan, Chuchu},
  journal={IEEE Transactions on Robotics},
  volume={39},
  number={3},
  pages={1749--1767},
  year={2023},
  publisher={IEEE}
}

@inproceedings{nguyen2024tinympc,
  title={Tinympc: Model-predictive control on resource-constrained microcontrollers},
  author={Nguyen, Khai and Schoedel, Sam and Alavilli, Anoushka and Plancher, Brian and Manchester, Zachary},
  booktitle={2024 IEEE International Conference on Robotics and Automation (ICRA)},
  pages={1--7},
  year={2024},
  organization={IEEE}
}

@article{erseghe2014distributed,
  title={Distributed optimal power flow using ADMM},
  author={Erseghe, Tomaso},
  journal={IEEE transactions on power systems},
  volume={29},
  number={5},
  pages={2370--2380},
  year={2014},
  publisher={IEEE}
}

@article{bagherzadeh2021guaranteed,
  title={Guaranteed collision-free reference tracking in constrained multi unmanned vehicle systems},
  author={Bagherzadeh, Maryam and Savehshemshaki, Shima and Lucia, Walter},
  journal={IEEE Transactions on Automatic Control},
  volume={67},
  number={6},
  pages={3083--3089},
  year={2021},
  publisher={IEEE}
}

@inproceedings{bishop2024relu,
  title={Relu-qp: A gpu-accelerated quadratic programming solver for model-predictive control},
  author={Bishop, Arun L and Zhang, John Z and Gurumurthy, Swaminathan and Tracy, Kevin and Manchester, Zachary},
  booktitle={2024 IEEE International Conference on Robotics and Automation (ICRA)},
  pages={13285--13292},
  year={2024},
  organization={IEEE}
}

@article{wensing2023optimization,
  title={Optimization-based control for dynamic legged robots},
  author={Wensing, Patrick M and Posa, Michael and Hu, Yue and Escande, Adrien and Mansard, Nicolas and Del Prete, Andrea},
  journal={IEEE Transactions on Robotics},
  volume={40},
  pages={43--63},
  year={2023},
  publisher={IEEE}
}

@article{panier1987superlinearly,
  title={A superlinearly convergent feasible method for the solution of inequality constrained optimization problems},
  author={Panier, Eliane R and Tits, Andr{\'e} L},
  journal={SIAM Journal on control and Optimization},
  volume={25},
  number={4},
  pages={934--950},
  year={1987},
  publisher={SIAM}
}

@article{wachter2006implementation,
  title={On the implementation of an interior-point filter line-search algorithm for large-scale nonlinear programming},
  author={W{\"a}chter, Andreas and Biegler, Lorenz T},
  journal={Mathematical programming},
  volume={106},
  number={1},
  pages={25--57},
  year={2006},
  publisher={Springer}
}

@article{scutari2016parallel,
  title={Parallel and distributed methods for constrained nonconvex optimization—Part I: Theory},
  author={Scutari, Gesualdo and Facchinei, Francisco and Lampariello, Lorenzo},
  journal={IEEE Transactions on Signal Processing},
  volume={65},
  number={8},
  pages={1929--1944},
  year={2016},
  publisher={IEEE}
}

@techreport{gill1986some,
  title={Some Theoretical Properties of an Augmented Lagrangian Merit Function.},
  author={Gill, Philip E and Murray, Walter and Saunders, Michael A and Wright, Margaret H},
  year={1986}
}

@article{nemirovski2008interior,
  title={Interior-point methods for optimization},
  author={Nemirovski, Arkadi S and Todd, Michael J},
  journal={Acta Numerica},
  volume={17},
  pages={191--234},
  year={2008},
  publisher={Cambridge university press}
}

@article{nocedal2009adaptive,
  title={Adaptive barrier update strategies for nonlinear interior methods},
  author={Nocedal, Jorge and W{\"a}chter, Andreas and Waltz, Richard A},
  journal={SIAM Journal on Optimization},
  volume={19},
  number={4},
  pages={1674--1693},
  year={2009},
  publisher={SIAM}
}

@book{wachter2002interior,
  title={An interior point algorithm for large-scale nonlinear optimization with applications in process engineering},
  author={Wachter, Andreas},
  year={2002},
  publisher={Carnegie Mellon University}
}

@article{schittkowski1983convergence,
  title={On the convergence of a sequential quadratic programming method with an augmented Lagrangian line search function},
  author={Schittkowski, Klaus},
  journal={Mathematische Operationsforschung und Statistik. Series Optimization},
  volume={14},
  number={2},
  pages={197--216},
  year={1983},
  publisher={Taylor \& Francis}
}

@article{panier1993combining,
  title={On combining feasibility, descent and superlinear convergence in inequality constrained optimization},
  author={Panier, Eliane R and Tits, Andr{\'e} L},
  journal={Mathematical programming},
  volume={59},
  number={1},
  pages={261--276},
  year={1993},
  publisher={Springer}
}

@book{nocedal2006numerical,
  title={Numerical optimization},
  author={Nocedal, Jorge and Wright, Stephen J},
  year={2006},
  publisher={Springer}
}

@article{muehlebach2021constraints,
  title={On constraints in first-order optimization: A view from non-smooth dynamical systems},
  author={Muehlebach, Michael and Jordan, Michael I},
  journal={arXiv preprint arXiv:2107.08225},
  year={2021}
}

@article{lu2012sequential,
  title={Sequential convex programming methods for a class of structured nonlinear programming},
  author={Lu, Zhaosong},
  journal={arXiv preprint arXiv:1210.3039},
  year={2012}
}

@inproceedings{domahidi2013ecos,
  title={ECOS: An SOCP solver for embedded systems},
  author={Domahidi, Alexander and Chu, Eric and Boyd, Stephen},
  booktitle={2013 European control conference (ECC)},
  pages={3071--3076},
  year={2013},
  organization={IEEE}
}

@article{goulart2024clarabel,
  title={Clarabel: An interior-point solver for conic programs with quadratic objectives},
  author={Goulart, Paul J and Chen, Yuwen},
  journal={arXiv preprint arXiv:2405.12762},
  year={2024}
}

@article{diamond2016cvxpy,
  title={CVXPY: A Python-embedded modeling language for convex optimization},
  author={Diamond, Steven and Boyd, Stephen},
  journal={Journal of Machine Learning Research},
  volume={17},
  number={83},
  pages={1--5},
  year={2016}
}

@article{liu1995sensitivity,
  title={Sensitivity analysis in nonlinear programs and variational inequalities via continuous selections},
  author={Liu, Jiming},
  journal={SIAM Journal on Control and Optimization},
  volume={33},
  number={4},
  pages={1040--1060},
  year={1995},
  publisher={SIAM}
}

@article{andersson2019casadi,
  title={CasADi: a software framework for nonlinear optimization and optimal control},
  author={Andersson, Joel AE and Gillis, Joris and Horn, Greg and Rawlings, James B and Diehl, Moritz},
  journal={Mathematical Programming Computation},
  volume={11},
  number={1},
  pages={1--36},
  year={2019},
  publisher={Springer}
}

@article{ames2016control,
  title={Control barrier function based quadratic programs for safety critical systems},
  author={Ames, Aaron D and Xu, Xiangru and Grizzle, Jessy W and Tabuada, Paulo},
  journal={IEEE Transactions on Automatic Control},
  volume={62},
  number={8},
  pages={3861--3876},
  year={2016},
  publisher={IEEE}
}

@article{misra2022learning,
  title={Learning for constrained optimization: Identifying optimal active constraint sets},
  author={Misra, Sidhant and Roald, Line and Ng, Yeesian},
  journal={INFORMS Journal on Computing},
  volume={34},
  number={1},
  pages={463--480},
  year={2022},
  publisher={Informs}
}

@article{wu2025accelerated,
  title={Accelerated Successive Convex Approximation for Nonlinear Optimization-Based Control},
  author={Wu, Jinxian and Dai, Li and Dou, Songshi and Xia, Yuanqing},
  journal={IEEE Transactions on Automatic Control},
  year={2025},
  publisher={IEEE}
}

@article{torrisi2018projected,
  title={A projected gradient and constraint linearization method for nonlinear model predictive control},
  author={Torrisi, Giampaolo and Grammatico, Sergio and Smith, Roy S and Morari, Manfred},
  journal={SIAM Journal on Control and Optimization},
  volume={56},
  number={3},
  pages={1968--1999},
  year={2018},
  publisher={SIAM}
}

@incollection{mayne2009surperlinearly,
  title={A surperlinearly convergent algorithm for constrained optimization problems},
  author={Mayne, David Q and Polak, Elijah},
  booktitle={Algorithms for constrained minimization of smooth nonlinear functions},
  pages={45--61},
  year={2009},
  publisher={Springer}
}

@article{byrd1987trust,
  title={A trust region algorithm for nonlinearly constrained optimization},
  author={Byrd, Richard H and Schnabel, Robert B and Shultz, Gerald A},
  journal={SIAM Journal on Numerical Analysis},
  volume={24},
  number={5},
  pages={1152--1170},
  year={1987},
  publisher={SIAM}
}

@article{numerow2024inherently,
  title={Inherently robust suboptimal mpc for autonomous racing with anytime feasible sqp},
  author={Numerow, Logan and Zanelli, Andrea and Carron, Andrea and Zeilinger, Melanie N},
  journal={IEEE Robotics and Automation Letters},
  volume={9},
  number={7},
  pages={6616--6623},
  year={2024},
  publisher={IEEE}
}

@article{tenny2004nonlinear,
  title={Nonlinear model predictive control via feasibility-perturbed sequential quadratic programming},
  author={Tenny, Matthew J and Wright, Stephen J and Rawlings, James B},
  journal={Computational Optimization and Applications},
  volume={28},
  number={1},
  pages={87--121},
  year={2004},
  publisher={Springer}
}

@inproceedings{dietrich2024nonconvex,
  title={Nonconvex scenario optimization for data-driven reachability},
  author={Dietrich, Elizabeth and Devonport, Alex and Arcak, Murat},
  booktitle={6th Annual Learning for Dynamics \& Control Conference},
  pages={514--527},
  year={2024},
  organization={PMLR}
}

@article{fukushima1986successive,
  title={A successive quadratic programming algorithm with global and superlinear convergence properties},
  author={Fukushima, Masao},
  journal={Mathematical Programming},
  volume={35},
  number={3},
  pages={253--264},
  year={1986},
  publisher={Springer}
}

@article{fukushima2003sequential,
  title={A sequential quadratically constrained quadratic programming method for differentiable convex minimization},
  author={Fukushima, Masao and Luo, Zhi-Quan and Tseng, Paul},
  journal={SIAM Journal on Optimization},
  volume={13},
  number={4},
  pages={1098--1119},
  year={2003},
  publisher={SIAM}
}

@article{west1992generalized,
  title={A generalized quadratic programming-based phase I-phase II method for inequality-constrained optimization},
  author={West, EJ and Polak, Elijah},
  journal={Applied Mathematics and Optimization},
  volume={26},
  number={3},
  pages={223--252},
  year={1992},
  publisher={Springer}
}

@book{bonnans2013perturbation,
  title={Perturbation analysis of optimization problems},
  author={Bonnans, J Fr{\'e}d{\'e}ric and Shapiro, Alexander},
  year={2013},
  publisher={Springer Science \& Business Media}
}

@article{mestres2025off,
  title={Off-Policy Reinforcement Learning with Anytime Safety Guarantees via Robust Safe Gradient Flow},
  author={Mestres, Pol and Marzabal, Arnau and Cort{\'e}s, Jorge},
  journal={arXiv preprint arXiv:2510.01492},
  year={2025}
}


\begin{IEEEbiography}
[{\includegraphics[width=1in,height=1.25in,clip,keepaspectratio]{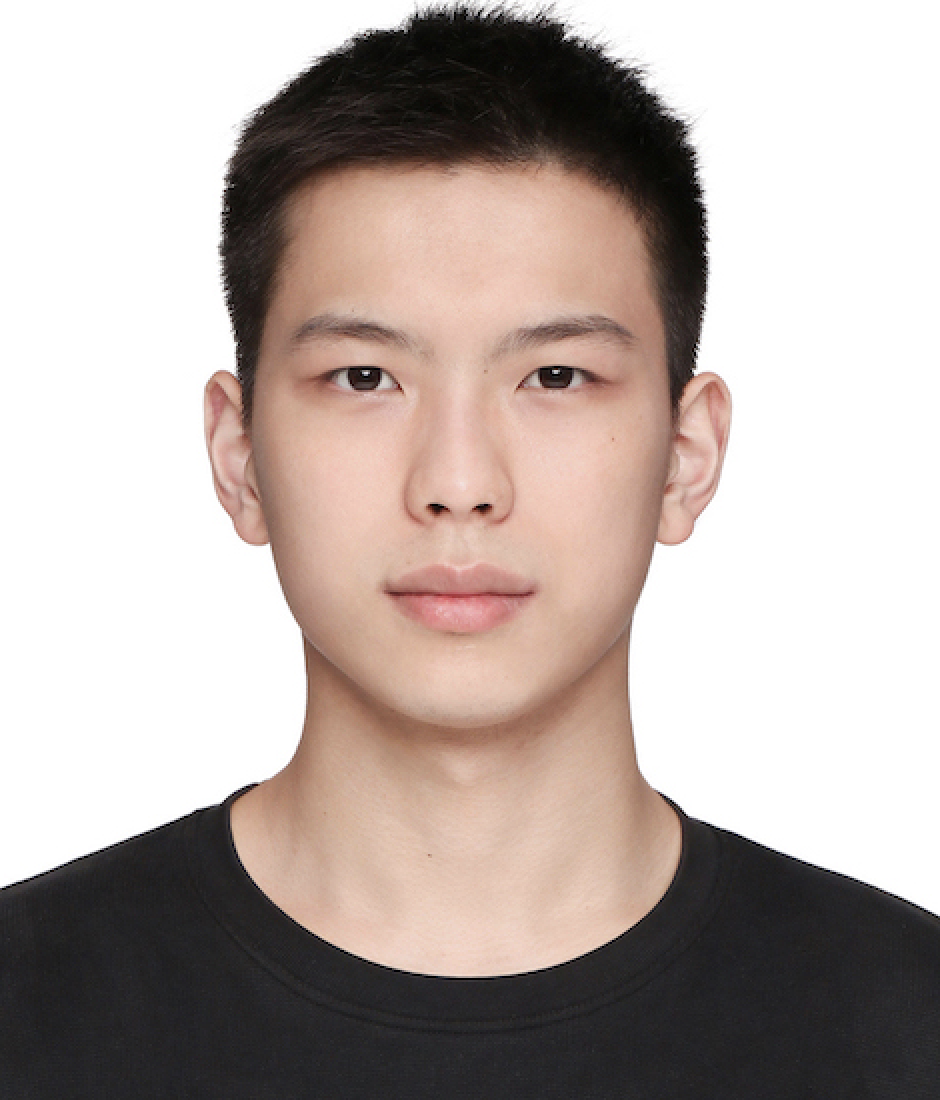}}]{Jiarui Wang}
received the B.S. degree in software engineering from Nanjing Forestry University, Nanjing, China, in 2021, and the M.S. degree in computer science from Johns Hopkins University, Baltimore, MD, USA, in 2023. He is currently pursuing a Ph.D. degree in electrical and computer engineering at Johns Hopkins University. His research interests include optimization, optimal control, and reinforcement learning.
\end{IEEEbiography}

\begin{IEEEbiography}
[{\includegraphics[width=1in,height=1.25in,clip,keepaspectratio]{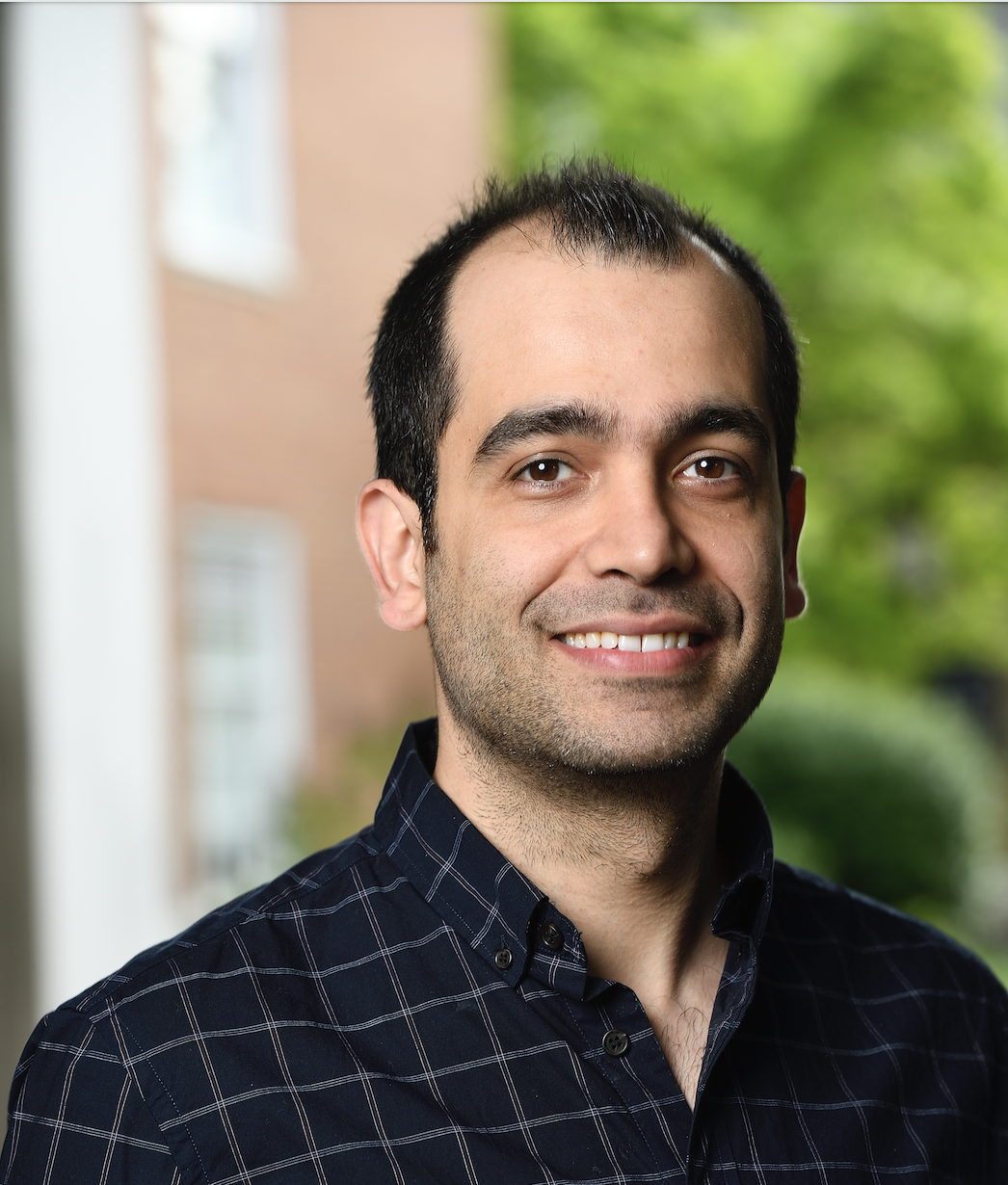}}]
{\bf Mahyar Fazlyab} (Member, IEEE) received the Ph.D. degree in Electrical and Systems Engineering from the University of Pennsylvania, Philadelphia, PA, USA, in 2018. He subsequently held a postdoctoral fellowship in the same department from 2018 to 2020. Since 2021, he has been an Assistant Professor in the Department of Electrical and Computer Engineering at Johns Hopkins University, where he is also affiliated with the Mathematical Institute for Data Science (MINDS) and holds a secondary appointment in the Department of Computer Science.

His research interests lie at the intersection of optimization, control, and machine learning, with a focus on safe and scalable algorithms for learning-enabled autonomous systems. Dr. Fazlyab is the recipient of the Joseph and Rosaline Wolf Best Doctoral Dissertation Award, presented by the University of Pennsylvania’s Department of Electrical and Systems Engineering in 2019.
\end{IEEEbiography}

\end{document}